\documentclass[notitlepage,11pt,reqno]{amsart}

\usepackage{amsopn,esint,nicefrac}
\usepackage{mathabx}
\usepackage[final]{hyperref}
\usepackage[T1]{fontenc}
\usepackage{graphicx}

\usepackage[utf8]{inputenc}
\usepackage{apptools}
\AtAppendix{\counterwithin{lemma}{section}} 
\usepackage[margin=1in]{geometry}
\allowdisplaybreaks
\newcommand{\stkout}[1]{\ifmmode\text{\sout{\ensuremath{#1}}}\else\sout{#1}\fi}

\newcommand{\R}{\mathbb{R}}

\newcommand{\Rn}{\mathbb{R}^n}
\usepackage{graphicx,enumitem,dsfont,upgreek}
 \usepackage[dvips]{epsfig}
 \usepackage[mathscr]{eucal}
\usepackage{amscd}
\usepackage{amssymb,nicefrac}
\usepackage{amsthm}
\usepackage{amsmath}
\usepackage{latexsym}
\usepackage{dsfont}
\usepackage{upref}
\usepackage{hyperref}

\usepackage{color}
\theoremstyle{plain}

\newtheorem{thm}{Theorem}[section]
\theoremstyle{plain}
\newtheorem{lem}[thm]{Lemma}

\newtheorem{defi}[thm]{Definition}
\newtheorem{rem}{Remark}[section]
\theoremstyle{definition}
\newtheorem*{maintheorem*}{Main Theorem}
\newtheorem*{maincorollary*}{Main Corollary}
{%
\setcounter{enumi}{0}

\begin{enumerate}}%
{\end{enumerate} }

{%
\setcounter{enumi}{0}

\begin{enumerate}}%
{\end{enumerate} }

\newcommand{\norm}[1]{\ensuremath{\left\|#1\right\|}}

\newcommand{\cone}{\ensuremath{\mathcal{C}}}
\newcommand{\cD}{\ensuremath{\mathcal{D}}}
\newcommand{\sL}{\ensuremath{\mathscr{L}}}

\newcommand{\cJ}{\ensuremath{\mathcal{J}}}

\newcommand{\dist}{{\rm dist}}
\newcommand{\xbar}{\ensuremath{\bar{x}}}
\newcommand{\ybar}{\ensuremath{\bar{y}}}
\newcommand{\abar}{\ensuremath{\bar{a}}}
\newcommand{\tail}{{\rm Tail}}

\newcommand{\dx}{\ensuremath{\, dx}}

\newcommand{\dz}{\ensuremath{\, dz}}

\numberwithin{equation}{section} \allowdisplaybreaks

\usepackage{cancel,pdfsync}

\title[Improved H\"{o}lder regularity]{Improved H\"{o}lder regularity of fractional $(p,q)$-Poisson
equation with regular data}

\begin{document}

\author{Anup Biswas}

\author{Aniket Sen}

\address{Indian Institute of Science Education and Research-Pune, Dr.\ Homi Bhabha Road, Pashan, Pune 411008, INDIA. Emails:
{\tt anup@iiserpune.ac.in, aniket.sen@students.iiserpune.ac.in}}

\begin{abstract}
We prove a quantitative H\"{o}lder continuity result for viscosity solutions to the equation
$$
(-\Delta_p)^{s}u(x) + {\rm PV} \int_{\Rn} |u(x)-u(x+z)|^{q-2}(u(x)-u(x+z))\frac{\xi(x,z)}{|z|^{n+ tq}}\dz=f \quad \text{in}\; B_2,
$$
where $t, s\in (0, 1), 1<p\leq q, tq\leq sp$ and $\xi\geq 0$. Specifically, we show that if $\xi$ is $\alpha$-H\"{o}lder continuous and
$f$ is $\beta$-H\"{o}lder continuous then any viscosity solution is locally $\gamma$-H\"{o}lder continuous for any  
$\gamma<\gamma_\circ $, where 
\[
\gamma_\circ=\left\{\begin{array}{lll}
\min\{1, \frac{sp+\alpha\wedge\beta}{p-1}, \frac{sp}{p-2}\} & \text{for}\; p>2,
\\
\min\{1, \frac{sp+\alpha\wedge\beta}{p-1}\} & \text{for}\; p\in (1, 2].
\end{array}
\right.
\]
Moreover, if 
$\min\{\frac{sp+\alpha\wedge\beta}{p-1}, \frac{sp}{p-2}\}>1$ when $p>2$, or  $\frac{sp+\alpha\wedge\beta}{p-1}>1$ when $p\in (1, 2]$, the solution is 
locally Lipschitz. This extends the result of \cite{DFP19} to the case of H\"{o}lder continuous modulating coefficients.
 Additionally, due to the equivalence between viscosity and weak solutions, our result provides a local Lipschitz estimate for weak solutions of $(-\Delta_p)^{s}u(x)=0$ provided either $p\in (1, 2]$ or $sp>p-2$ when $p>2$, thereby improving recent works
\cite{BDLMS24a,BDLMS24b,DKLN}.
\end{abstract}


\maketitle


\section{Introduction}

In this article, we study a class of nonlocal double phase equations which are possibly singular and degenerate, integro-differential equations and its leading operator switches between two different
fractional elliptic phases according to the zero set of the modulating coefficient $\xi$. More precisely, our operator is
given by 
\begin{equation*}
\sL u(x)=(-\Delta_p)^{s}u(x) + {\rm PV} \int_{\Rn} |u(x)-u(x+z)|^{q-2}(u(x)-u(x+z))\frac{\xi(x,z)}{|z|^{n+tq}}\dz,
\end{equation*}
where $(-\Delta_p)^{s}$ denotes the well known fractional $p$-Laplacian, given by,
$$(-\Delta_p)^{s}u(x)={\rm PV}\int_{\Rn}|u(x)-u(x+z)|^{p-2}(u(x)-u(x+z))\frac{1}{|z|^{n+ sp}}\dz.$$
Here PV stands for the principal value. We impose the following conditions on the parameters:
\begin{equation}\label{para-con}
s, t\in (0, 1),\quad 1< p\leq q<\infty,\quad \text{and}\quad t q\leq sp.
\end{equation}
Moreover, we assume that $\xi$ is symmetric in the second variable, that is, $\xi(x, z)=\xi(x, -z)$ for all $z\in\Rn$ and bounded
$$0\leq \xi\leq M.$$
We emphasize that this symmetry assumption on $\xi$ differs from the one considered in the variational setting, see \cite{BOS22,DKP14,DKP16}, and can be interpreted as a {\it non-divergence} form of the nonlocal double phase problem. The above framework is more suitable to work with viscosity
solutions, see also \cite{Lin16}. Furthermore, if we impose additional conditions on $\xi$-- such as symmetry (i.e., $\xi(x, y)=\xi(y, x)$) and translation invariance (i.e., $\xi(x+z,y+z)=\xi(x, y)$)-- then, as shown in \cite{FZ23},  bounded weak solutions of 
\eqref{E2.1} are H\"{o}lder continuous and also viscosity solution to \eqref{E2.1}. Thus, under these additional hypotheses,  our result extends to bounded weak solutions. 

The operator $\sL$ can be viewed as a nonlocal counterpart to classical double phase problems, arising from the minimization of the functional
\begin{equation}\label{AB001}
\mathcal{E}(u)=\int (|\nabla u|^p + \xi(x) |\nabla u|^q) \dx\quad 1<p\leq q.
\end{equation}
This functional emerges in homogenization studies \cite{Z86,Z95} and belongs to the class of non-uniformly elliptic problems with $(p,q)$-growth conditions. While the regularity theory for general $(p, q)$-growth problems is highly challenging, the specific structure of \eqref{AB001} 
 has been extensively investigated--see \cite{BCM18,CM15,CM15a} and the survey \cite{MR21}.

In this work, we focus on viscosity solutions $u$ to the equation
\begin{equation}\label{E2.1}
\sL u=f \quad \text{in}\; B_2,
\end{equation}
where $B_r$ denotes the ball of radius $r$ centered at the origin. The precise definition of a viscosity solution will be introduced in the next section
( see Definition~\ref{Def1.1}). To state our main result, let us introduce the \textit{tail spaces}.
By $L^{p-1}_{sp}(\Rn)$ we denote the weighted $L^p$ space  defined by
$$L^{p-1}_{sp}(\Rn)=\left\{f\in L^{p-1}_{\rm loc}(\Rn)\; :\; \int_{\Rn}\frac{|f(z)|^{p-1}}{(1+|z|)^{n+sp}}\dz <\infty\right\}.$$
Associated to this tail space we also define the {\it tail function} given by
$$
\tail_{s,p}(f; x, r) := \left(r^{sp} \int_{|z-x|\geq r}\frac{|f(z)|^{p-1}}{|z-x|^{n+sp}}\dz \right)^{\frac{1}{p-1}}\quad r>0.
$$
Analogously, we define $L^{q-1}_{tq}(\Rn)$ and $\tail_{t,q}(f; x, r)$. Let
$$A=\sup_{B_2}|u| + \tail_{s,p}(u; 0, 2) + \tail_{t,q}(u; 0, 2).$$
Our main result of this article is the following
\begin{thm}\label{Tmain-1}
Let $p\in (1, \infty)$ and $\xi:\bar{B}_2\times\Rn\to[0, \infty)$ be $\alpha$-H\"{o}lder continuous, that is,
$$|\xi(x_1, z_1)-\xi(x_2, z_2)|\leq C(|x_1-x_2|^\alpha + |z_1-z_2|^\alpha)\quad \text{for}\; (x_1, z_1), \, (x_2, z_2)\in \bar{B}_2\times\Rn.$$
Also, let $f\in C^{0,\beta}(\bar{B}_2)$
and $u\in C(B_2)\cap L^{p-1}_{sp}(\Rn)\cap L^{q-1}_{tq}(\Rn)$ be a viscosity solution to \eqref{E2.1}. 
Then $u\in C^{0, \gamma}(\bar{B}_1)$ for any $\gamma<\gamma_\circ $, where 
\[
\gamma_\circ=\left\{\begin{array}{lll}
\min\{1, \frac{sp+\alpha\wedge\beta}{p-1}, \frac{sp}{p-2}\} & \text{for}\; p>2,
\\[2mm]
\min\{1, \frac{sp+\alpha\wedge\beta}{p-1}\} & \text{for}\; p\in (1, 2].
\end{array}
\right.
\]
Moreover, if $\gamma_\circ=\frac{sp+\alpha\wedge\beta}{p-1}<\min\{1, \frac{sp}{p-2}\}$ when $p>2$, or 
$\gamma_\circ=\frac{sp+\alpha\wedge\beta}{p-1}<1$ when $p\in (1, 2]$, we have  $u\in C^{0, \gamma_\circ}(\bar{B}_1)$. Also,
if $\min\{\frac{sp+\alpha\wedge\beta}{p-1}, \frac{sp}{p-2}\}>1$ when $p>2$, or $\frac{sp+\alpha\wedge\beta}{p-1}>1$ when $p\in (1, 2]$,
we have $u\in C^{0, 1}(\bar{B}_1)$.
 In all the above cases, 
$\norm{u}_{C^{0, \gamma}(\bar{B}_1)}$ (including $\gamma=1, \gamma_\circ$) is bounded by a constant $C$, dependent on $M, A, s, t, p, q, n$, $\norm{\xi}_{C^{0,\alpha}(B_2\times\Rn)}$ and $\norm{f}_{C^{0, \beta}(B_2)}$.
\end{thm}

\begin{rem}
If $\xi$ is only assumed to be $\alpha$-H\"{o}lder in the first argument uniformly with respect to the second argument,
and $f$ is merely bounded, then from our proofs below it can be shown that $u\in C^{0, \tilde\gamma}(\bar{B}_1)$ where $\tilde\gamma=\min\{1, \frac{sp}{p-1}\}$ and $\frac{sp}{p-1}\neq 1$. 
\end{rem}

\begin{rem}
Unlike the estimates in \cite{BDLMS24a,BDLMS24b,BLS18,GL24}, our bounds in this work are not stable as $s\nearrow 1$ or $t\nearrow 1$. The
reason is purely technical. For example, the estimate in Lemma~\ref{L3.2} and in the proof of Theorem~\ref{T3.7} (see \eqref{ET3.7A}) depend on
$s\in (0,1)$ and we do not have a uniform choice of constant as $s\nearrow 1$ or $t\nearrow 1$.
\end{rem}

By leveraging the equivalence between weak and viscosity solutions, along with Theorem~\ref{Tmain-1}, we obtain the following result, which generalizes recent findings in \cite{BT25,BDLMS24a,BDLMS24b}:

\begin{thm}\label{T1.2}
Let $u\in W^{s, p}_{\rm loc}(B_2)\cap L^{p-1}_{sp}(\Rn)$ be a local weak solution to $(-\Delta_p)^s u=f$ in $B_2$ and $f\in C^{0,\beta}_{\rm loc}(B_2)$. Then 
$u\in C^{0, \gamma}(\bar{B}_1)$ for any $\gamma<\gamma_\circ$, where 
\[
\gamma_\circ=\left\{\begin{array}{ll}
\min\{1, \frac{sp+\beta}{p-1}, \frac{sp}{p-2}\} & \text{for}\; p>2,
\\[2mm]
\min\{1, \frac{sp+\beta}{p-1}\} & \text{for}\; p\in (1, 2].
\end{array}
\right.
\] 
Moreover, if $\gamma_\circ=\frac{sp+\beta}{p-1}<\min\{1, \frac{sp}{p-2}\}$ when $p>2$, or 
$\gamma_\circ=\frac{sp+\beta}{p-1}<1$ when $p\in (1, 2]$, we have  $u\in C^{0, \gamma_\circ}(\bar{B}_1)$.
Additionally,
if $\min\{\frac{sp+\beta}{p-1}, \frac{sp}{p-2}\}>1$ when $p>2$, or $\frac{sp+\beta}{p-1}>1$ when $p\in (1, 2]$,
we have $u\in C^{0, 1}(\bar{B}_1)$. In particular, for $f=0$, $u$ is locally Lipschitz, provided either $p\in (1, 2]$ or $sp>p-2$ when $p>2$.
\end{thm}

The regularity theory for nonlocal operators has been an active area of research over the past two decades. Since the $W^{s, p}$ energy functional converges to the 
classical $p$-Dirichlet energy after a suitable normalization \cite{BBM01}, the fractional $p$-Laplacian emerges as a natural operator for investigation.
In the context of the fractional $p$-Laplacian, qualitative H\"{o}lder estimates were initially obtained using De Giorgi-Nash-Moser theory, see \cite{APT24,BP16,Coz17,DKP16,KKP16}. Lindgren \cite{Lin16}
employed a Krylov-Safonov type argument in the viscosity setting to establish $\gamma$-H\"{o}lder regularity for some small $\gamma$. Additionally,  
Di Castro, Kuusi and Palatucci derived a Harnack estimate in \cite{DKP14}. 
A major breakthrough in quantitative H\"{o}lder estimates came with the work of Brasco, Lindgren and
Schikorra \cite{BLS18} (see also Brasco and Lindgren \cite{BL17}), who proved optimal $\gamma$-H\"{o}lder regularity in the superquadratic case $p\geq 2$. 
Specifically, for $\xi=0$ and $f\in L^\infty$, they showed that a local weak solution of \eqref{E2.1} is $\gamma$-H\"{o}lder for any $\gamma<\breve\gamma=\min\{1, \frac{sp}{p-1}\}$. 
A corresponding result for the subquadratic case ($1<p<2$) was later established by Garain and Lindgren \cite{GL24}. This marked the beginning of extensive research on higher-order fractional differentiability and H\"{o}lder regularity of solutions, including \cite{BDLM25,DKLN,DN25} and references
therein.  Fine zero-order estimates were investigated in \cite{KMS15a,KMS15b}, while boundary regularity was explored in \cite{IMS16,IMS20,KLL23,KKP16}. 
Recently, the first author and Topp \cite{BT25} proved
that the exponent $\breve\gamma$ is attained when $\frac{sp}{p-1}\neq 1$. In particular, for $\frac{sp}{p-1}>1$, local weak solutions of $(-\Delta_p)^s u=f$ ($f\in L^\infty$) are shown to be locally Lipschitz.

This article is partly motivated by the recent works of B\"{o}gelein et al. \cite{BDLMS24a,BDLMS24b} and
Diening et al. \cite{DKLN}, where the authors study fractional $p$-harmonic functions-- that is, solutions to $(-\Delta_p)^su=0$, 
and prove that such $u$ is almost $\widecheck{\gamma}$-H\"{o}lder, where $\widecheck\gamma=\min\{1, \frac{sp}{p-2}\}$ when $p>2$, and $\widecheck\gamma=1$ when $p\in (1, 2]$. Notably,
$\widecheck\gamma\geq \breve\gamma$, indicating a gain in regularity for fractional $p$-harmonic functions compared to the general case. A natural question arises: 
is $\widecheck{\gamma}$ attained? In particular, do fractional
 $p$-harmonic functions become Lipschitz continuous when $\widecheck\gamma\geq 1$? Theorem ~\ref{T1.2} answers this question in a more general setting. Another motivation for our work stems from the study of nonlocal double phase problems. De Filippis and Palatucci \cite{DFP19} were the first to investigate the regularity of \eqref{E2.1} for bounded 
 $\xi$ and $f$, establishing $\gamma$-H\"{o}lder regularity of bounded solutions for some $\gamma>0$. Their parameters required a slightly stronger condition than \eqref{para-con} (see  \cite[(2.4)--(2.5)]{DFP19}). Subsequently,  Scott and Mengesha \cite{SM22} proved a nonlocal self-improving property for bounded weak solutions, while
 Fang and Zhang \cite{FZ23} established $\gamma$-H\"{o}lder continuity of bounded solutions under the condition \eqref{para-con} with $f=0$. We also highlight the work of 
Byun, Ok and Song \cite{BOS22}, where the authors considered a different set of conditions (that is, $s\leq t$, $p\leq q, tq\leq sp+\alpha$) and proved
$\gamma$-H\"{o}lder regularity of solutions when $f=0$ and $\xi$ is $\alpha$-H\"{o}lder continuous. While the aforementioned results are of qualitative nature, this article provides a quantitative estimate, aligning with the findings of \cite{BDLMS24a,BDLMS24b}. Moreover, our analysis does not require solutions to be bounded. We also point out that Theorem~\ref{Tmain-1} answers a question raised in \cite[p.~551]{DFP19} on the higher regularity
of nonlocal double phase problems.

For the proof of Theorem~\ref{Tmain-1}, we follow the approach of \cite{BT25}, which relies on an Ishii-Lions-type argument. As previously mentioned, for $f\in L^\infty$, \cite{BT25} establishes that
the solutions of $(-\Delta_p)^s u=f$ are in $C^{0, \breve\gamma}_{\rm loc}$, provided $sp\neq p-1$. In this work, we exploit the H\"{o}lder continuity of $f$ to further enhance the regularity of solutions
to \eqref{E2.1}.  Our proof thus requires refined estimates of several key quantities compared to \cite{BT25}, particularly in Step 3 and Step 4 of the proofs of Theorem~\ref{T3.7} and ~\ref{T4.3}.
The Ishii-Lions argument was first introduced in the seminal work of Ishii and Lions \cite{IL90},
where the authors established H\"{o}lder regularity for solutions to a class of nondegenerate elliptic equations. This approach was later adapted to nonlocal operators in the work of Barles et al. \cite{BCI11} (see also \cite{BCCI12}). Since then, the Ishii-Lions-type argument has been successfully applied to a wide range of nonlinear nonlocal equations, as evidenced by \cite{BT24,BQT25,CGT22},
among others. 

The remainder of this article is organized as follows: In the next section,  we introduce the concept of viscosity solutions and outline the general strategy for the proof. In Section~\ref{s-sup}
covers the superquadratic case ($p>2$), while Section~\ref{s-sub} addresses the subquadratic case ($1<p\leq 2$). Proofs of Theorems~\ref{Tmain-1} and ~\ref{T1.2} are presented in Section~\ref{s-sub}.

Throughout the paper, we use the notations $C, C_1, C_2, .., \kappa,\kappa_1, \kappa_2,...$ to denote generic constants whose values may vary from line to line.

\section{General strategy for the proof}
We primarily focus on proving Theorem~\ref{Tmain-1}, as the proof of Theorem~\ref{T1.2} follows directly from the equivalence between weak and viscosity solutions. The solution to \eqref{E2.1} is understood in the viscosity sense, which we define below in the spirit of  \cite{KKL19}, see also \cite{FZ23}. 
First, we recall some notation from 
\cite{KKL19}.  Since, as established in \cite{KKL19}, the operator $\sL$ may not be classically defined for all $C^2$ functions, we must restrict our consideration to a suitable subclass of test functions when defining viscosity solutions.
Given an open set $D$, we denote by $C^2_\eta(D)$, a subset of $C^2(D)$, defined as
$$
C^2_\eta(D)=\left\{\phi\in C^2(D)\; :\; \sup_{x\in D}\left[\frac{\min\{d_\phi(x), 1\}^{\eta-1}}{|\nabla\phi(x)|} +
\frac{|D^2\phi(x)|}{(d_\phi(x))^{\eta-2}}\right]<\infty\right\},
$$
where
$$ 
d_\phi(x)=\dist(x, N_\phi)\quad \text{and}\quad N_\phi=\{x\in D\; :\; \nabla\phi(x)=0\}.$$

The above restricted class of test functions becomes necessary to define $\sL$ in the classical sense in the singular case, that is, for
$p\leq \frac{2}{2-s}$ or $q\leq \frac{2}{2-t}$. 
Now we are ready to define the viscosity solution from \cite[Definition~3]{KKL19}. Let us also point out that under the above stated condition \eqref{para-con} we have $\frac{sp}{p-1}\geq \frac{tq}{q-1}, \frac{2}{2-t}\leq \frac{2}{2-s}$ and
we use this fact implicitly while defining the viscosity solution.

\begin{defi}\label{Def1.1}
A function $u:\Rn\to \R$ is a viscosity subsolution to $\sL u =f $ in $\Omega$ if it satisfies the following
\begin{itemize}
\item[(i)] $u$ is upper semicontinuous in $\bar\Omega$.
\item[(ii)] If $\phi\in C^2(B_r(x_0))$ for some $B_r(x_0)\subset \Omega$ satisfies $\phi(x_0)=u(x_0)$,
$\phi\geq u$ in $B_r(x_0)$ and one of the following holds
\begin{itemize}
\item[(a)] $\min\{p-\frac{2}{2-s}, q-\frac{2}{2-t}\}>0$ or $\nabla\phi(x_0)\neq 0$,

\item[(b)] $\min\{p-\frac{2}{2-s}, q-\frac{2}{2-t}\}\leq 0$, $\nabla\phi(x_0)= 0$ is such that $x_0$ is an isolated critical point of $\phi$, and
$\phi\in C^2_\eta(B_r(x_0))$ for some $\eta>\frac{sp}{p-1}$,
\end{itemize}
then we have $\sL \phi_r(x_0)\leq f(x_0)$ where
\[
\phi_r(x)=\left\{\begin{array}{ll}
\phi(x) & \text{for}\; x\in B_r(x_0),
\\[2mm]
u(x) & \text{otherwise}.
\end{array}
\right.
\]

\item[(iii)] $u_+\in L^{p-1}_{sp}(\Rn)\cap L^{q-1}_{tq}(\Rn)$.
\end{itemize}
We say $u$ is a viscosity supersolution in $\Omega$, if $-u$ is a viscosity subsolution in $\Omega$. Furthermore, a 
viscosity solution of $\sL u= f$ in $\Omega$ is both sub and super solution in $\Omega$. 
\end{defi}

It is clear from the structure of $\sL$ that the above definition is equivalent to general testing, in the sense that $x_0$ is a local maximum point for $u - \phi$ (and not necessarily $u(x_0) = \phi(x_0)$). Same comment goes for the testing of supersolutions.
Let us now introduce the notations: $J_p(t)=|t|^{p-2} t$, $J_q(t)=|t|^{q-2} t$,
\begin{equation}\label{AB01}
\sL_p[D] w(x):=  {\rm PV}\int_D J_p(w(x)-w(x+z))\frac{dz}{|z|^{n+sp}},
\end{equation}
and
\begin{equation}\label{AB02}
\sL_q[D] w(x):=  {\rm PV}\int_D J_q(w(x)-w(x+z))\xi(x,z)\frac{dz}{|z|^{n+tq}}.
\end{equation}
With these notations we can modify the definition of viscosity solution as follows.
\begin{lem}\label{L2.2}
Let $\phi\in C^2(B_r(x_0))$ be a test function that touches $u$ from above in $B_r(x_0)$ at the point $x_0$ and either (a) or (b) in
Definition~\ref{Def1.1} holds. Let $\delta, \tilde\delta\in (0, r)$ and define
\[
v_\delta(x)=\left\{ \begin{array}{ll}
\phi(x) & \text{for}\; x\in B_\delta(x_0),
\\[2mm]
u(x) & \text{otherwise},
\end{array}
\right.
\quad and\quad 
v_{\tilde\delta}(x)=\left\{ \begin{array}{ll}
\phi(x) & \text{for}\; x\in B_{\tilde\delta}(x_0),
\\[2mm]
u(x) & \text{otherwise}.
\end{array}
\right.
\]
Then we have
$$\sL_p[\Rn] v_\delta(x_0)+\sL_q[\Rn]v_{\tilde\delta}(x_0)\leq f(x_0).$$
Analogous result holds for supersolutions.
\end{lem}

\begin{proof}
For  $\hat\delta=\delta\wedge\tilde\delta$ and
\[
\hat{v}(x):=\left\{ \begin{array}{ll}
\phi(x) & \text{for}\; x\in B_{\hat\delta}(x_0),
\\[2mm]
u(x) & \text{otherwise},
\end{array}
\right.
\]
we know from Definition~\ref{Def1.1} that
$\sL \hat{v}(x_0)= \sL_p[\Rn] \hat{v}(x_0) + \sL_q[\Rn] \hat{v}(x_0)\leq f(x_0)$. Now note that
$v_\delta(x_0)-v_\delta(x_0+z)= u(x_0)-v_\delta(x_0+z)\leq \hat{v}(x_0)-\hat{v}(x_0+z)$ for all $z$. Similarly, 
$v_{\tilde\delta}(x_0)-v_{\tilde\delta}(x_0+z)\leq \hat{v}(x_0)-\hat{v}(x_0+z)$ for all $z$. Thus, using the monotonicity of
$J_p$ and $J_q$, we obtain $\sL_p[\Rn] v_\delta(x_0)+\sL_q[\Rn]v_{\tilde\delta}(x_0)\leq f(x_0)$.
\end{proof}

\subsection{General strategy}\label{S-genstr}
In  this section we explain our strategy for the proofs.
Fix $1\leq \varrho_1<\varrho_2\leq 2$, and define the doubling
function
\begin{equation}\label{E2.2}
\Phi(x, y)= u(x)-u(y)-L\varphi(|x-y|)- m_1 \psi(x)\quad x, y\in B_2,
\end{equation}
where 
$$
\psi(x) = [(|x|^2-\varrho^2_1)_+]^{m}\quad   x \in B_2,
$$ 
is a {\it localization} function. We set $m\geq 3$ so that $\psi\in C^2(B_2)$. 

The function $\varphi:[0, 2] \to [0, \infty)$ is a suitable {\it regularizing} function, encoding the modulus of continuity for $u$. 

For calculations below, we use two types of regularizing function $\varphi$ (after an appropriate scaling):
\begin{equation}\label{varphi}
\begin{split}
 \varphi_\gamma(t) &=t^{\gamma}\quad \mbox{with} \ \gamma\in (0, 1) \quad \mbox{(for H\"older profile)}, 
 \\
 \tilde\varphi(t) &=\left\{\begin{array}{ll}
t+\frac{t}{\log t}& \text{for}\; x>0,
\\
0 & \text{for}\; x = 0.
\end{array}
\right.
 \quad  \mbox{(for Lipschitz profile)}.
\end{split}
\end{equation}
Notice that the above functions are increasing and concave on a neighbourhood of $t = 0$. 

We show that for a sufficiently large $m_1, m$,
and for all $L$ large enough, dependent on $p, q, s, t, A, M$, we have $\Phi\leq 0$ in $B_2\times B_2$, which leads to the desired result.

We proceed by contradiction, assuming that $\sup_{B_2\times B_2}\Phi>0$ for all large $L$. Let us choose $m_1$ large enough, dependent on
$\varrho_2, \varrho_1$ and $m$, so that $m_1 \psi(x)> 2A$ for all $|x|\geq \frac{\varrho_2+\varrho_1}{2}$. Then for all 
$|x|\geq \frac{\varrho_2+\varrho_1}{2}$, we have $\Phi(x, y)<0$ for all $y\in B_2$. Again, since $\varphi$ is strictly increasing in
$[0, 2]$, if we choose $L$ to satisfy $L\varphi(\frac{\varrho_2-\varrho_1}{4})>2A$, we obtain $\Phi(x, y)<0$ whenever
$|x-y|\geq \frac{\varrho_2-\varrho_1}{4}$. Thus, there exist $\xbar\in B_{\frac{\varrho_2+\varrho_1}{2}}$ and 
$\ybar\in B_{\frac{3\varrho_2}{4}+\frac{\varrho_1}{4}}$ such that
\begin{equation}\label{E2.3}
\sup_{B_2\times B_2}\Phi=\Phi(\xbar,\ybar)>0.
\end{equation}
Denote by $\abar=\xbar-\ybar$. From \eqref{E2.3} we have $\abar\neq 0$, and moreover, we have that
\begin{equation}\label{AB03}
L \varphi(|\bar a|) \leq u(\bar x) - u(\bar y) \leq 2A,
\end{equation}
in view of~\eqref{E2.2}. This implies that $|\bar a|$ is small as $L$ enlarges.
Let us denote by
$$
\phi(x, y) := L\varphi(|x-y|)+ m_1 \psi(x).
$$
Note that
\begin{center}
$x\mapsto u(x) - \phi(x, \ybar)$ has a local maximum point at $\xbar$, and\\[2mm]
$y \mapsto u(y) +\phi(\xbar,y)$ has a local minimum point at $\ybar$.
\end{center}
For $\delta,\tilde\delta\in (0, \frac{\varrho_2-\varrho_1}{4})$ to be chosen later, we define the following test functions

\[
w_1(z)=\left\{\begin{array}{ll}
\phi(z, \ybar) + \kappa_{\xbar} & \text{if}\; z\in B_\delta(\bar x),
\\[2mm]
u(z) & \text{otherwise},
\end{array}
\right.
\quad \text{and}\quad
\tilde{w}_1(z)=\left\{\begin{array}{ll}
\phi(z, \ybar) +\kappa_{\xbar} & \text{if}\; z\in B_{\tilde\delta}(\bar x),
\\[2mm]
u(z) & \text{otherwise},
\end{array}
\right.
\]
with $\kappa_{\xbar}=u(\xbar)- \phi(\xbar, \ybar)$, and
\[
w_2(z)=\left\{\begin{array}{ll}
-\phi(\xbar, z)+ \kappa_{\ybar} & \text{if}\; z\in B_\delta(\bar y),
\\[2mm]
u(z) & \text{otherwise},
\end{array}
\right.
\quad \text{and}\quad
\tilde{w}_2(z)=\left\{\begin{array}{ll}
-\phi(\xbar, z)+ \kappa_{\ybar} & \text{if}\; z\in B_{\tilde\delta}(\bar y),
\\[2mm]
u(z) & \text{otherwise},
\end{array}
\right.
\]
with $\kappa_{\ybar}= u(\ybar) + \phi(\xbar, \ybar)$.

An important point here is that, regardless of the choice of $\varphi$ above, for all sufficiently large
$L$ (depending on $m$ and $m_1$), we must have $\nabla_x\phi(\xbar,\ybar)\neq 0$ and $\nabla_y\phi(\xbar,\ybar)\neq 0$. Thus, from 
Lemma~\ref{L2.2}
and \eqref{E2.1}, we obtain
$$
\sL_p[\Rn] w_1(\bar x) + \sL_q[\Rn] \tilde{w}_1(\bar x)\leq f(\xbar)\quad \text{and}\quad 
\sL_p[\Rn] w_2 (\bar y) + \sL_q[\Rn]\tilde{w}_2(\ybar) \geq f(\ybar).
$$
As can be seen from \cite{KKL19}, the above principal values are well defined.
Subtracting the viscosity inequalities at $\bar x$ and $\bar y$, we obtain
\begin{equation}\label{E2.4}
\sL_p[\Rn] w_1(\bar x)-\sL_p[\Rn] w_2(\bar y) + \sL_q[\Rn] \tilde{w}_1(\xbar) -\sL_q[\Rn] \tilde{w}_2(\ybar)
\leq f(\bar x) - f(\bar y) \leq C |\bar a|^{\beta}.
\end{equation}
We define the following domains
$$\cone=\{z\in B_{\delta_0|\abar|}\; :\; |\langle \abar, z\rangle|\geq (1-\eta_0) |\abar||z|\},
\quad \cD_1=B_\delta \cap \cone^c, \quad \text{and}\quad \cD_2=B_{\tilde\varrho}\setminus (\cD_1\cup\cone),$$
where $\delta_0=\eta_0\in (0, \frac{1}{2})$ would be chosen later, $\tilde\varrho=\frac{1}{4}(\varrho_2-\varrho_1)$, and, in general, $\delta << \delta_0 |\bar a| << \tilde \varrho$. See Figure~\ref{figure} .
\begin{figure}[ht]
\begin{center}
\includegraphics[scale=0.24]{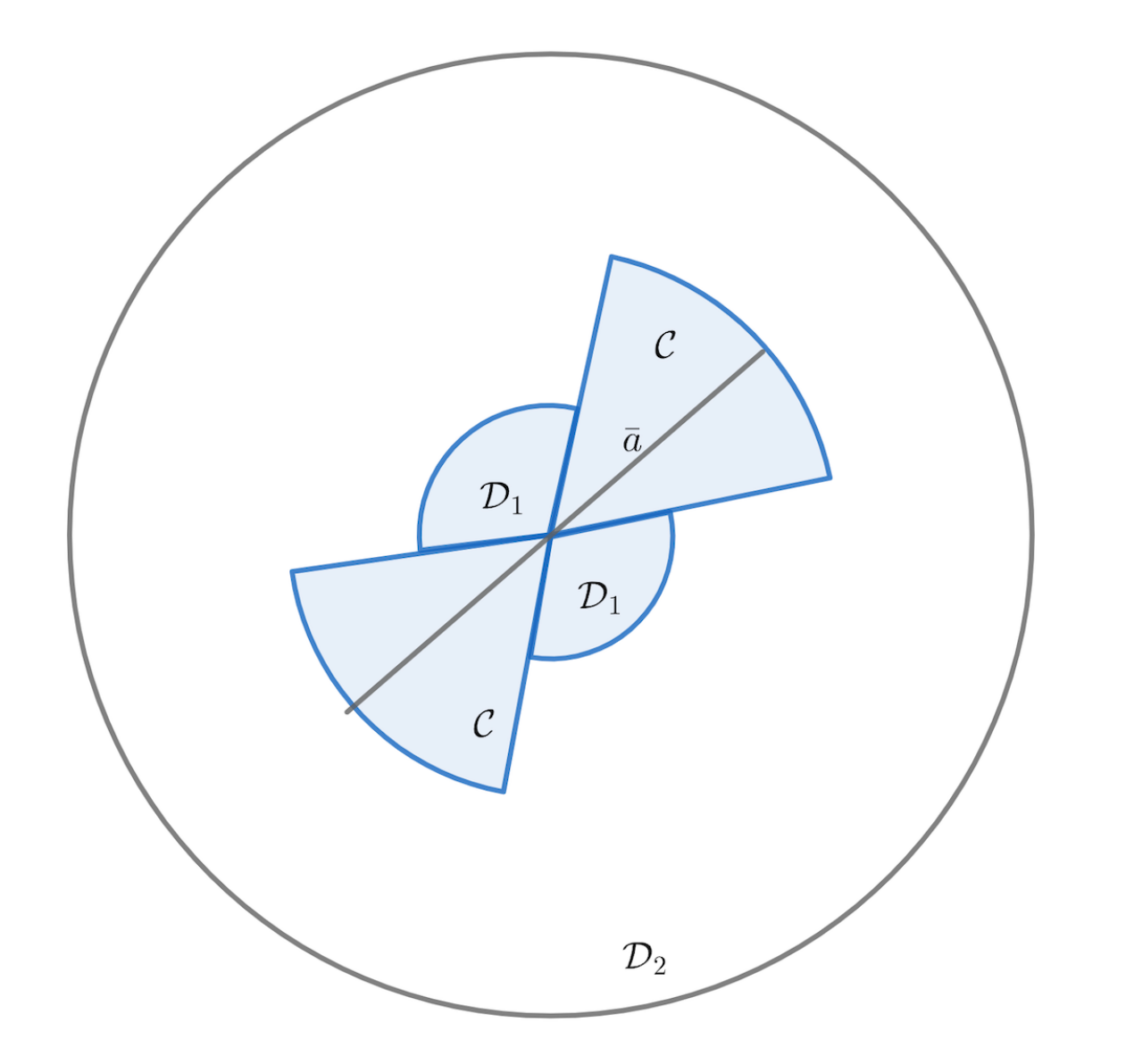}
\end{center}
\caption{Relevant sets in the nonlocal viscosity evaluation.}
\label{figure}
\end{figure}
Also, let $\tilde{\cD}_2= B_{\tilde\varrho}\setminus B_{\tilde\delta}$. Using \eqref{AB01}, \eqref{AB02} and \eqref{E2.4} we arrive at
\begin{align}\label{E2.7}
&\underbrace{\sL_p[\cone] w_1(\bar x)-\sL_p[\cone] w_2(\bar y)}_{=I_1} + \underbrace{\sL_p[\cD_1] w_1(\bar x)-\sL_p[\cD_1] w_2(\bar y)}_{=I_2} + \nonumber
\\ 
& \underbrace{\sL_p[\cD_2] w_1(\bar x)-\sL_p[\cD_2] w_2(\bar y)}_{=I_3}
+ \underbrace{\sL_p[B^c_{\tilde\varrho}] w_1(\bar x)-\sL_p[ B^c_{\tilde\varrho}] w_2(\bar y)}_{=I_4} + \underbrace{\sL_q[B_{\tilde\delta}] \tilde{w}_1(\bar x) -\sL_q[B_{\tilde\delta}] \tilde{w}_2(\bar y)}_{=J_1} \nonumber
\\
&  + \underbrace{\int_{\tilde{\cD}_2}[J_q(\tilde{w}_1(\xbar) - \tilde{w}_1(\xbar + z))-J_q(\tilde{w}_2(\ybar) - \tilde{w}_2(\ybar + z))]\frac{\xi(\xbar, z)}{|z|^{n+tq}}\dz}_{=J_2} + \nonumber
\\
& \underbrace{\int_{\tilde{\cD}_2}\left[J_q(\tilde{w}_2(\ybar) - \tilde{w}_2(\ybar + z))\right]\frac{\xi(\xbar, z)-\xi(\ybar, z)}{|z|^{z+tq}}\dz}_{=J_3} + 
 \underbrace{\sL_q[B^c_{\tilde\varrho}] \tilde{w}_1(\bar x)-\sL_q[ B^c_{\tilde\varrho}] \tilde{w}_2(\bar y)}_{=J_4} \leq C |\abar|^{\beta}.
\end{align}
Our main goal is to estimate the terms $I_i, J_i, i=1,2,3,4$ suitable so that \eqref{E2.7} leads to contradiction for all $L$ large enough. 

We conclude the section by recalling the following algebraic estimate from
\cite{KKL19}.

\begin{lem}\label{L2.4}
Suppose $p>1$ and $a, b\in \R$. Then 
\begin{equation*}
\frac{1}{c_p} (|a|+|b|)^{p-2}\leq \int_0^1 |a+tb|^{p-2} dt\leq c_p (|a|+|b|)^{p-2},
\end{equation*}
for some constant $c_p$, depending only on $p$. 
\end{lem}

\section{Proof in the superquadratic case $p>2$}\label{s-sup}
We begin with the estimates of $I_i, i=1,.., 4$. The following estimate can be found in \cite{BT25}

\begin{lem}[Estimate of $I_1$]\label{L3.1}
Let $p\in (1, \infty)$.
For $0<|\abar|\leq \frac{1}{8}$, consider the cone 
$\cone=\{z\in B_{\delta_0|\abar|}\; :\; |\langle \abar, z\rangle| \geq (1-\eta_0) |\abar||z|\}$, where $\delta_0=\eta_0\in (0, 1)$.
Then 
\begin{itemize}
\item[(i)] For $\varphi(t)=\varphi_\gamma(t)=t^\gamma$, $\gamma\in (0, 1)$, there exist $L_0, \delta_0$, dependent on $m, m_1, \gamma$, such that
$$I_1\geq C L^{p-1}|\abar|^{\gamma(p-1)-sp}$$
for all $L\geq L_0$, where the constant $C$ depends on $\delta_0, p, s, \gamma, n$.

\item[(ii)] Let $r_0>0$ be small enough so that for $r\in (0, r_0]$ we have
\begin{gather*}
\frac{r}{2}\leq \tilde\varphi(r)\leq r, \quad \frac{1}{2}\leq \tilde\varphi'(r)\leq 1,
\\
-2(r\log^2(r))^{-1}\leq\tilde\varphi^{\prime\prime}(r)\leq -(r\log^2(r))^{-1}.
\end{gather*}
Letting $\varphi(t) = \tilde \varphi(\frac{r_\circ}{3}t)$,  there exist $L_0, \delta_1$, independent of $\abar$,
such that for $\delta_0=\delta_1(\log^2|\abar|)^{-1}$ we have
$$I_1\geq C L^{p-1} |\abar|^{p-1-sp} (\log^2(|\abar|))^{-\uptheta}$$
for all $L\geq L_0$, where $\uptheta=\frac{n+1}{2}+p-sp$ and the constant $C$ depends only on $\delta_1, p, s, n$.
\end{itemize}
\end{lem}

\begin{proof}
First let $p\geq 2$. Using \cite[Lemma~3.1]{BT25} we can find $L_0$ such that
$$I_1\geq \kappa (L\varphi'(|a|))^{p-2}\int_{\cone} |z|^{p-2-n-sp} (\triangle^2\phi(\cdot, \ybar)(\xbar, z) + \triangle^2\phi(\xbar, \cdot)(\ybar, z)))\dz$$
for all $L\geq L_0$, where $\triangle^2 g(x, z):= g(x) - g(x+z) +\nabla g(x)\cdot z$. Now use \cite[Lemma~2.2]{BT25} which provides  estimates on
$\triangle^2\phi(\cdot, \ybar)(\xbar, z)$ and $\triangle^2\phi(\xbar, \cdot)(\ybar, z)$ together with the estimate of the integral (see \cite[Example~1]{BCCI12})
$$\int_{\cone}|z|^{p-n-sp} \dz\gtrsim \frac{\rm Vol(\cone)}{\rm Vol(B_{\delta_0|\abar|})}\int_{B_{\delta_0|\abar|}}|z|^{p-n-sp} \dz\gtrsim \frac{1}{p(1-s)} \eta_0^{\frac{n-1}{2}}(\delta_0|\abar|)^{p-sp}.$$
Readers may refer to the estimate (3.9) in \cite{BT25} for additional detail.

In case of $p\in (1, 2)$, we use \cite[Lemma~4.1]{BT25} and repeat the same argument as above.
\end{proof}

The following estimate of $I_2$ follows from \cite[Lemmas~3.2 and ~4.2]{BT25}.
\begin{lem}[Estimate of $I_2$]\label{L3.2}
Let $p\in (1, \infty)$ and $\delta=\varepsilon_1|\abar|$ for $\varepsilon_1\in (0, \nicefrac{1}{2})$. Then there exist $C, L_0$, independent of $\varepsilon_1, |\abar|$, such that
$$I_2 \geq - C L^{p-1} \varepsilon_1^{p(1-s)} (\varphi'(|\abar|))^{p-1} |\abar|^{p(1-s)-1},$$
where $\varphi$ is given by \eqref{varphi}. Moreover, if we set $\delta =\varepsilon_1 (\log^{2\rho}(|\abar|))^{-1}|\abar|$ with
$\rho=\frac{\frac{n+1}{2}+p-sp}{p-sp}$, and let $\varphi(t)=\tilde\varphi(\frac{r_0}{3}t)$ from Lemma~\ref{L3.1}(ii),
then we have $L_0, C>0$ satisfying
$$I_2 \geq - C L^{p-1} \varepsilon_1^{p(1-s)}  |\abar|^{p(1-s)-1} (\log^2(|\abar|))^{-\uptheta},$$
for all $L\geq L_0$.
\end{lem}

Next we estimate $I_3$ and then $I_4$.
\begin{lem}[Estimate of $I_3$]\label{L3.3}
Suppose that $p> 2$ and $u\in C^{0, \upkappa}(\bar{B}_{\varrho_2})$ for some $\upkappa\in [0, \min\{\frac{sp}{p-2}, 1\})$. Let $\delta=\delta_0|\abar|$ where $\delta_0$
is chosen as in Lemma~\ref{L3.1}.
Then for any $\theta\in (0, 1)$, we have $L_0>0$ such that
$$I_3\geq -C \left[\int_\delta^{|\abar|^\theta} r^{\upkappa(p-2) + 1-sp} dr  + 
|\abar|^{\frac{m-1}{m}\upkappa} \int_\delta^{|\abar|^\theta} r^{\upkappa(p-2) -sp} dr + |\abar|^{\kappa +\theta(\upkappa(p-2)-sp)}\right]$$
for all $L\geq L_0$, where the constant $C$ depends on $\upkappa, p, s, n, m, m_1$ and the $C^{0, \upkappa}$ norm of $u$ in
$B_{\varrho_2}$.
\end{lem}

\begin{proof}
Fix $\theta\in (0, 1)$. From \eqref{AB03} we see that $|\abar|\to 0$ as $L\to\infty$. Therefore, we would have $\delta<|\abar|^\theta<\tilde\varrho$ for all
$L$ large enough. We denote $\hat\delta=|\abar|^\theta$ and write
$$
I_3=\underbrace{\sL_p[\cD_2\cap B_{\hat\delta}] w_1(\xbar)-\sL_p[\cD_2\cap B_{\hat\delta}] w_2(\ybar)}_{I_{1, 3}} +
\underbrace{\sL_p[\cD_2\cap B^c_{\hat\delta}] w_1(\xbar)-\sL_p[\cD_2\cap B^c_{\hat\delta}] w_2(\ybar)}_{I_{2, 3}}.
$$
Using the notation $\triangle g(x, z)= g(x)-g(x+z)$ and fundamental theorem of calculus we see that
$$I_{1,3}=(p-1)\int_{\cD_2\cap B_{\hat\delta}}\int_0^t |\triangle u(\ybar, z) + t (\triangle u(\xbar, z)-\triangle u(\ybar, z))|^{p-2}
(\triangle u(\xbar, z)-\triangle u(\ybar, z))\, dt \frac{dz}{|z|^{n+sp}}.$$
Now $\Phi(\xbar+z,\ybar+z)\leq \Phi(\xbar,\ybar)$ gives us
$$\triangle u(\xbar, z) - \triangle u(\ybar, z)\geq m_1 \triangle \psi(\xbar, z),$$
leading to
\begin{equation}\label{EL3.3A}
I_{1,3}\geq -\kappa\, m_1 \int_{\cD_2\cap B_{\hat\delta}} (|\triangle u(\ybar, z)| +  |\triangle u(\xbar, z)|)^{p-2} |\triangle \psi(\xbar, z)|\frac{dz}{|z|^{n+sp}}
\end{equation}
for some constant $\kappa$, where we also used Lemma~\ref{L2.4}. Note that $\xbar+z, \ybar+z\in B_{\varrho_2}$ for all $z\in B_{\tilde\varrho}$.
Therefore,
$$|\triangle u(\ybar, z)| +  |\triangle u(\xbar, z)|\leq 2 [u]_{\upkappa, \varrho_2} |z|^\upkappa$$
where $[u]_{\upkappa, \varrho_2}$ denotes the $C^{0, \upkappa}$ seminorm in $B_{\varrho_2}$. Again, from the Taylor's expansion
of $\psi$ we also get
$$|\triangle \psi(\xbar, z)|\leq \kappa_2 (|z|^2 + |\nabla\psi(\xbar)||z|)$$
for some constant $\kappa_2$, dependent on $m$. Putting these estimates in \eqref{EL3.3A} we arrive at
\begin{align}\label{EL3.3B}
I_{1,3} &\geq -\kappa_3 \int_{B^c_{\delta}\cap B_{\hat\delta}} (|z|^{\upkappa(p-2) + 2} + |\nabla\psi(\xbar)| |z|^{\upkappa(p-2) + 1}) \frac{dz}{|z|^{n+sp}}
\nonumber
\\
& =-\kappa_4 \int_\delta^{\hat\delta} (r^{\upkappa(p-2) + 2} + |\nabla\psi(\xbar)| r^{\upkappa(p-2) + 1}) r^{-1-sp} dr
\end{align}
for some constants $\kappa_3, \kappa_4$, dependent on $[u]_{\upkappa, \varrho_2}, m, m_1, n$. From \eqref{E2.3} we have
$$\psi(\xbar)\leq \frac{1}{m_1} (u(\xbar)-u(\ybar))\leq \frac{1}{m_1} [u]_{\upkappa, \varrho_2} |\abar|^\upkappa.$$
Thus
$$|\nabla\psi(\xbar)|\leq 2m (\psi(\xbar))^{\frac{m-1}{m}}\leq \kappa_5 |\abar|^{\frac{m-1}{m}\upkappa}$$
for some constant $\kappa_5$. Hence from \eqref{EL3.3B} we obtain
\begin{equation}\label{EL3.3C}
I_{1,3} \geq -C \left[ \int_\delta^{\hat\delta} r^{\upkappa(p-2) + 1-sp} dr  + 
|\abar|^{\frac{m-1}{m}\upkappa} \int_\delta^{\hat\delta} r^{\upkappa(p-2) -sp} dr\right].
\end{equation}
To compute $I_{2,3}$ we first note, using Lemma~\ref{L2.4}, that
\begin{align*}
I_{2,3} &\geq -\kappa_p \int_{\cD_2\cap B^c_{\hat\delta}} (|\triangle u(\ybar, z)|  + | \triangle u(\xbar, z)|)^{p-2}
|\triangle u(\xbar, z)-\triangle u(\ybar, z)| \frac{dz}{|z|^{n+sp}}
\\
&\geq -\kappa_6 \int_{\cD_2\cap B^c_{\hat\delta}} |z|^{\upkappa(p-2)} |\abar|^{\upkappa} \frac{dz}{|z|^{n+sp}}
\\
&\geq -\kappa_7 |\abar|^{\kappa} \int_{\hat\delta}^{\tilde\rho} r^{\upkappa(p-2)-sp-1} dr
\\
&\geq - \frac{\kappa_7}{sp-\upkappa(p-2)} |\abar|^{\kappa} (\hat\delta)^{\upkappa(p-2)-sp}=
- \frac{\kappa_7}{sp-\upkappa(p-2)} |\abar|^{\kappa +\theta(\upkappa(p-2)-sp)} 
\end{align*}
for some constants $\kappa_6, \kappa_7$. The proof follows by combining the above estimate with \eqref{EL3.3C}.
\end{proof}

\begin{lem}[Estimate of $I_4$ and $J_4$]\label{L3.4}
Let $p, q\in (1, \infty)$. Suppose that $u\in C^{0, \upkappa}(B_{\varrho_2})$ for some $\upkappa\in [0, 1]$ and 
$\xi$ is $\alpha$-H\"{o}lder continuous. Then
there is a constant $L_0$ such that
$$|I_4|\leq C \max\{|\abar|^\upkappa, |\abar|^{\upkappa(p-1)}\} \quad \text{and}\quad
 |J_4|\leq C \max\{|\abar|^\alpha, |\abar|^{\upkappa}, |\abar|^{\upkappa(q-1)}\}$$
for all $L\geq L_0$, where the constant $C$ depends on $\tilde\varrho, A$, $\sup_z\norm{\xi(\cdot, z)}_{C^{0, \alpha}}$ and the $C^{0, \upkappa}$ norm of $u$ in $B_{\varrho_2}$.
\end{lem}

\begin{proof}
We only provide a proof for $J_4$ as the proof for $I_4$ would follow in a similar fashion. Denote $\zeta(x, z)=\xi(x, z)|z|^{-n-tq}$.
Then estimate
\begin{align}\label{EL3.4A}
|J_5| &\leq \underbrace{ \left| \int_{|\xbar-z|\geq \tilde\varrho} J_q(u(\xbar)-u(z)) |\zeta(\xbar, \xbar-z)-\zeta(\ybar, \ybar-z)| \dz\right|}_{=\cJ_1}\nonumber
\\
&\quad + \underbrace{\left|\int_{|\xbar-z|\geq \tilde\varrho} J_q(u(\xbar)-u(z)) \zeta(\ybar, \ybar-z) \dz- \int_{|\ybar-z|\geq \tilde\varrho} J_q(u(\xbar)-u(z)) \zeta(\ybar, \ybar-z) \dz\right|}_{=\cJ_2}\nonumber
\\
&\qquad + \underbrace{\left|\int_{|\ybar-z|\geq \tilde\varrho} (J_q(u(\xbar)-u(z)) - J_q(u(\ybar)-u(z)))\zeta(\ybar, \ybar-z) \dz\right|}_{=\cJ_3}.
\end{align}
Set $L_0$ large enough using \eqref{AB03}, so that $|\abar|<\frac{\tilde\varrho}{2}$. Thus, we get 
$|\xbar-z|\geq \varrho\Rightarrow |\ybar-z|\geq \frac{\tilde\varrho}{2}$. Therefore, using
the H\"{o}lder continuity of $\xi$, we obtain for $|\xbar-z|\geq\frac{\tilde\varrho}{2}$ that
\begin{align*}
|\zeta(\xbar, \xbar-z)-\zeta(\ybar, \ybar-z)| &\leq |z-\xbar|^{-n-tq} |\xi(\xbar, \xbar-z)-\xi(\ybar, \ybar-z)|+ \xi(\ybar, \ybar-z) \left||z-\xbar|^{-n-tq}-|z-\ybar|^{-n-tq}\right|
\\
&\leq \kappa \left(|z-\xbar|^{-n-tq} |\abar|^{\alpha} + |z-\xbar|^{-n-1-tq}|\abar|\right)
\\
&\leq \kappa_1 |\abar|^{\alpha} |z-\xbar|^{-n-tq}
\end{align*}
for some constants $\kappa, \kappa_1$. Using the finiteness of $A$, it follows that $\cJ_1\leq C |\abar|^\alpha$.
To estimate $\cJ_2$, we first note that we only need to calculate the integration on the set  
\begin{align*}
&(\{|\xbar-z|\geq \tilde\varrho\}\setminus\{|\ybar-z|\geq \tilde\varrho\})\cup (\{|\ybar-z|\geq \tilde\varrho\}\setminus\{|\xbar-z|\geq \tilde\varrho\})
\\
&\quad = (\{|\xbar-z|\leq \tilde\varrho\}\setminus\{|\ybar-z|\leq \tilde\varrho\})\cup (\{|\ybar-z|\leq \tilde\varrho\}\setminus\{|\xbar-z|\leq \tilde\varrho\}).
\end{align*}
Since $\sup_{B_2}|u|\leq 1$, it follows that $\cJ_2\leq C |\abar|$. To compute $\cJ_3$, we see that, for $q> 2$, 
$$
|J_q(u(\xbar)-u(z)) - J_q(u(\ybar)-u(z))|
\leq \kappa_2 (|u(z)|+1)^{q-2} |u(\xbar)-u(\ybar)|\leq \kappa_3(|u(z)|+1)^{q-2} |\abar|^{\upkappa}
$$
for some constants $\kappa_2, \kappa_3$. Hence, by our assumption on $\tail_{t,q}(u;0,2)$, we obtain
\begin{align*}
\cJ_3 &\leq \kappa_3 M |\abar|^{\upkappa}\int_{|\ybar-z|\geq \tilde\varrho}\frac{(|u(z)|+1)^{q-2}}{|z-\ybar|^{n+tq}}\dz
\\
&\leq \kappa_3 M |\abar|^{\upkappa} \left[ \int_{|\ybar-z|\geq \tilde\varrho} \frac{(|u(z)|+1)^{q-1}}{|z-\ybar|^{n+tq}}\dz \right]^{\frac{q-2}{q-1}}
	\left[ \int_{|\ybar-z|\geq \tilde\varrho} \frac{1}{|z-\ybar|^{n+tq}}\dz \right]^{\frac{1}{q-1}}
\\
&\leq \kappa_4 |\abar|^{\upkappa}
\end{align*}
for some constant $\kappa_4$. For $q\in (1,2]$, we use 
$$
|J_q(u(\xbar)-u(z)) - J_q(u(\ybar)-u(z))|
\leq 2 |u(\xbar)-u(\ybar)|^{q-1} \leq \kappa |\abar|^{\upkappa(q-1)},
$$
giving us $\cJ_3\leq \kappa_1 |\abar|^{\upkappa(q-1)}$. 
Plug-in these estimates in \eqref{EL3.4A} we have
\begin{equation*}
|J_5|\leq C \max\{|\abar|^{\alpha}, |\abar|^{\upkappa}, |\abar|^{\upkappa(q-1)}\}.
\end{equation*}
This completes the proof.
\end{proof}

Estimate of $J_1$ is similar to $I_2$ in Lemma~\ref{L3.2}
\begin{lem}[Estimate of $J_1$]\label{L3.5}
Let $q\in (1, \infty), p\leq q, tq\leq sp$ and $\tilde\delta=\varepsilon_1|\abar|$ for $\varepsilon_1\in (0, \nicefrac{1}{2})$. Then there exist $C, L_0$, independent of $\varepsilon_1, |\abar|$, such that
$$J_1 \geq - C L^{p-1} \varepsilon_1^{q(1-t)}  |\abar|^{\gamma(p-1)-sp},$$
where $\varphi=\varphi_\gamma$ is given by \eqref{varphi}.  Moreover, if we set $\tilde\delta=\varepsilon_1 (\log^{2\tilde\rho}(|\abar|))^{-1}|\abar|$ with
$\tilde\rho=\frac{\frac{n+1}{2}+q-tq}{q-tq}$, and let $\varphi(r)=\tilde\varphi(\frac{r_0}{3}r)$ from Lemma~\ref{L3.1}(ii),
then we have $L_0, C>0$ satisfying
$$J_1 \geq - C L^{p-1} \varepsilon_1^{q(1-t)}  |\abar|^{p(1-s)-1} (\log^2(|\abar|))^{-\uptheta},$$
for all $L\geq L_0$, where $\uptheta=\frac{n+1}{2}+p-sp$.
\end{lem}

\begin{proof}
Following the proofs of \cite[Lemma~3.2 and ~4.2]{BT25} we see that for $\tilde\delta=\tilde\delta_0 |\abar|$ with $\tilde\delta_0\in (0, \nicefrac{1}{2})$, we have
\begin{equation}\label{EL3.5A}
J_1\geq - \kappa (\tilde\delta_0)^{q-tq} (L\varphi'(|\abar|))^{q-1} |\abar|^{q(1-t)-1}
\end{equation}
for a constant $\kappa$, not depending on $|\abar|, L$ and $\tilde\delta_0$.  Now suppose that $\varphi=\varphi_\gamma$ and $\tilde\delta_0=\varepsilon_1$. Then
$$(\varphi'(|\abar|))^{q-1} |\abar|^{q(1-t)-1}=\gamma^{q-1} |\abar|^{\gamma(q-1) -tq}.$$
On the other have, \eqref{AB03} implies $L\leq 2A |\abar|^{-\gamma}\Rightarrow L^{q-p}\leq (2A |\abar|^{-\gamma})^{q-p}.$
Plugin this information in \eqref{EL3.5A} gives us 
$$J_1\geq -\kappa_3 L^{p-1} \varepsilon_1^{q(1-t)} |\abar|^{\gamma(p-q)} |\abar|^{\gamma(q-1) -tq}\geq -\kappa_3 L^{p-1} \varepsilon_1^{q(1-t)} |\abar|^{\gamma(p-1)-sp},$$
where in the last inequality we use that fact that $|\abar|\in (0, 1)$ and $sp\geq tq$. This gives the first inequality.

For the second estimate, observe that $\frac{1}{2}\leq \varphi'(|\abar|)\leq 1$, by our choice in Lemma~\ref{L3.1}. Again $\varphi(|\abar|)\geq \kappa |\abar|$ gives us
$L\leq \kappa_4 |\abar|^{-1}$ from \eqref{AB03}. Thus, setting $\tilde\delta_0=\varepsilon_1 (\log^{2\tilde\rho}(|\abar|))^{-1}$, we get from \eqref{EL3.5A} that
$$J_1\geq -\kappa_5 L^{p-1} \varepsilon_1^{q-tq} |\abar|^{p-q} |\abar|^{q(1-t)-1} (\log^2(|\abar|))^{-\tilde\uptheta}\geq 
-\kappa_5 L^{p-1} \varepsilon_1^{q-tq} |\abar|^{p(1-s)-1} (\log^2(|\abar|))^{-\tilde\uptheta},$$
where $\tilde\uptheta=\frac{n+1}{2}+q-tq$. By our assumption, we have $\tilde\uptheta\geq \uptheta$. Hence the desired estimate follows.
\end{proof}

Next we estimate $J_2$ and $J_3$.
\begin{lem}[Estimate of $J_2$ and $J_3$]\label{L3.6}
Let $q>2$ and $u\in C^{0, \upkappa}(\bar{B}_{\varrho_2})$ for some $\upkappa\in [0, \min\{\frac{sp}{p-2}, 1\})$. Let $\tilde\delta$ be given by Lemma~\ref{L3.5}.
Then for any $\theta\in (0, 1)$ we have we have $L_0>0$ such that
$$J_2\geq -C \left[\int_{\tilde\delta}^{|\abar|^\theta} r^{\upkappa(q-2) + 1- tq} dr  + 
|\abar|^{\frac{m-1}{m}\upkappa} \int_{\tilde\delta}^{|\abar|^\theta} r^{\upkappa(q-2) -tq} dr + |\abar|^{\kappa +\theta(\upkappa(p-2)-sp)}\right]$$
for all $L\geq L_0$, where the constant $C$ depends on $\upkappa, M, p, s, n, m, m_1$ and the $C^{0, \upkappa}$ norm of $u$ in
$B_{\varrho_2}$. Moreover, if $\xi$ is $\alpha$-H\"{o}lder continuous in the first argument uniformly with respect to the second, then
$$J_3\geq - C|\abar|^\alpha \int_{\tilde\delta}^{\tilde\varrho} r^{\upkappa(q-1)-1-tq} dr,$$
where the constant $C$ depends on $\sup_z\norm{\xi(\cdot, z)}_{C^{0, \alpha}}$ and the $C^{0, \upkappa}$ norm of $u$ in
$B_{\varrho_2}$.
\end{lem}

\begin{proof}
Estimate of $J_2$ follows from Lemma~\ref{L3.3}. Note that to estimate the term analogous to $I_{2,3}$ we only use the inequality
$$\int_{\hat\delta}^{\tilde\rho} r^{\upkappa(q-2)-tq-1} dr\leq \int_{\hat\delta}^{\tilde\rho} r^{\upkappa(p-2)-sp-1} dr.$$ 
For the second estimate, using the H\"{o}lder continuity of $\xi$, we see that
\begin{align*}
J_3\geq -\kappa |\abar|^\alpha\int_{\tilde\cD_2} |z|^{\upkappa(q-1)-n-tq} \dz = -\kappa_1 |\abar|^\alpha \int_{\tilde\delta}^{\tilde\varrho} r^{\upkappa(q-1)-1-tq} dr
\end{align*}
for some suitable constants $\kappa, \kappa_1$.
\end{proof}

Now we are ready to prove our main result in the superquadratic case.
\begin{thm}\label{T3.7}
Let $p>2$ and $\xi$ be $\alpha$-H\"{o}lder continuous. Also, let $f\in C^{0,\beta}(\bar{B}_2)$
and $u\in C(B_2)\cap L^{p-1}_{sp}(\Rn)\cap L^{q-1}_{tq}(\Rn)$ be a viscosity solution to \eqref{E2.1}. Then $u\in C^{0, \gamma}(\bar{B}_1)$ for any $\gamma<\gamma_\circ$ where
$\gamma_\circ=\min\{\frac{sp+\alpha\wedge\beta}{p-1}, \frac{sp}{p-2}, 1\}$. Additionally, when $\gamma_\circ=\frac{sp+\alpha\wedge\beta}{p-1}<\min\{1, \frac{sp}{p-2}\}$, we have
$u\in C^{0, \gamma_\circ}(\bar{B}_1)$, and
when 
$\min\{\frac{sp+\alpha\wedge\beta}{p-1}, \frac{sp}{p-2}\}>1$, we have $u\in C^{0, 1}(\bar{B}_1)$.
The H\"{o}lder norm $\norm{u}_{C^{0, \gamma}(B_1)}$, including $\gamma=1$, is bounded above  by a constant dependent on the data 
$M, A, s, p, t, q, n$.
\end{thm}

\begin{proof}
We break the proof into multiple steps as follows:
\begin{itemize}
\item[Step 1] --We show $u\in C^{0, \gamma}_{\rm loc}(B_2)$ for any $\gamma<\min\{1, \frac{sp}{p-1}\}=\breve{\gamma}$.
\item[Step 2]-- We show that $u \in C^{0, \breve\gamma}_{\rm loc}(B_2)$.
\item[Step 3]-- We show that $u \in C^{0, \gamma}_{\textit{\rm loc}}(B_2)$ for any $\gamma<\gamma_\circ$, and when $\gamma_{\circ}=\frac{sp+ \alpha \wedge \beta}{p-1} <\min \{ 1, \frac{sp}{p-2} \}$, we have $u \in C^{0, \gamma_\circ}_{\textit{\rm loc}}(B_2)$.
\item[Step 4]-- Assuming $\min\{\frac{sp+\alpha\wedge\beta}{p-1}, \frac{sp}{p-2}\} >1$, we show that $u \in C^{0,1}_{\rm loc}(B_2)$.
\end{itemize}

\noindent{\bf Step 1.}\; First, we show that $u\in C^{0, \gamma}_{\rm loc}(B_2)$ for any $\gamma<\min\{1, \frac{sp}{p-1}\}=\breve{\gamma}$. In fact, the proof in this step works if $\xi$ is 
$\alpha$-H\"{o}lder in the first coordinate uniformly with respect to the second. The proof uses a bootstrapping argument.
Fix $1<\varrho_1<\varrho_2<2$, $\gamma<\breve\gamma$ and let
$u\in C^{0, \upkappa}(B_{\varrho_2})$ for some $\upkappa\in [0, \gamma)$. Set
$$\upkappa_1=\min\{\gamma, \upkappa + \frac{1}{2(p-1)}, \upkappa + \frac{sp-(p-2)\gamma}{2(p-1)}, \upkappa + \frac{\alpha}{p-1}\}.$$
Note that $\gamma<\frac{sp}{p-1}<\frac{sp}{p-2}\Rightarrow sp-(p-2)\gamma>0$. Since $\upkappa_1< \upkappa+ \frac{1}{(p-1)}$,
we can choose $m\geq 3$ large enough so that 
$$ \upkappa(p-2) +1  \geq \upkappa_1 (p-1) - \frac{m}{m-1}\upkappa.$$
With this choice of $\varrho_1, \varrho_2, m$ and $\varphi=\varphi_{\upkappa_1}$ (see \eqref{varphi}) we consider \eqref{E2.2} and arrive at \eqref{E2.7}.
We set $\delta=\tilde\delta=\varepsilon_1|\abar|$. Using Lemma~\ref{L3.1}, ~\ref{L3.2} and ~\ref{L3.5}, we choose
$\varepsilon_1$ small enough, independent of $L$, so that
\begin{equation}\label{ET3.7A}
I_1+I_2+J_2\geq \frac{C}{2} L^{p-1} |\abar|^{\upkappa_1(p-1)-sp}
\end{equation}
for all $L\geq L_0$. Since $\upkappa_1\leq \upkappa + \frac{sp-(p-2)\gamma}{(p-1)}< \upkappa + \frac{sp-(p-2)\upkappa}{(p-1)}=\frac{sp+\upkappa}{p-1}$, we can choose $\theta\in (0, 1)$
small enough so that $\upkappa +\theta(\upkappa(p-2)-sp)>\upkappa_1(p-1)-sp$. To estimate $I_3$ we set this $\theta$ in Lemma~\ref{L3.3}. Since $|\abar|< 1$, we get
$$|\abar|^{\kappa +\theta(\upkappa(p-2)-sp)}\leq |\abar|^{\upkappa_1 (p-1)-sp}.$$
Again, by our choice of $m$, we have
$\upkappa (p-2)+1-sp\geq \upkappa_1 (p-1)-\frac{m}{m-1}\upkappa -sp$, implying
\begin{align*}
|\abar|^{\frac{m}{m-1}\upkappa}\int_{\delta}^{|\abar|^\theta} r^{\upkappa (p-2)-sp}dr
&\leq |\abar|^{\frac{m}{m-1}\upkappa}\int_{\delta}^{1} r^{\upkappa_1 (p-1)-\frac{m}{m-1}\upkappa-1 -sp}dr
\\
&= \frac{1}{sp-\upkappa_1 (p-1)+\frac{m}{m-1}\upkappa} \varepsilon_1^{\upkappa_1 (p-1)-\frac{m}{m-1}\upkappa -sp} |\abar|^{\upkappa_1 (p-1)-sp}
\\
&\leq \frac{1}{sp-\gamma (p-1)+\frac{m}{m-1}\upkappa} \varepsilon_1^{\upkappa_1 (p-1)-\frac{m}{m-1}\upkappa -sp} |\abar|^{\upkappa_1 (p-1)-sp}.
\end{align*}
Also, since $\upkappa(p-2) + 1 -sp> \upkappa_1(p-1) -\upkappa -sp> \upkappa_1(p-1) -1 -sp$, we have
$$\int_\delta^{|\abar|^\theta} r^{\upkappa(p-2) + 1-sp} dr\leq \int_\delta^{1} r^{\upkappa_1(p-1) -1 -sp} dr\leq \frac{\varepsilon_1^{\upkappa_1(p-1)-sp}}{sp-\gamma(p-1)}|\abar|^{\upkappa_1(p-1)-sp}.$$
Gathering the above estimate in Lemma~\ref{L3.3} we see that for some constant $C_{\varepsilon_1}$, dependent on $\varepsilon_1, \gamma, p, s$, we have
$$I_3\geq -C_{\varepsilon_1} |\abar|^{\upkappa_1(p-1)-sp}.$$
Since $\upkappa(q-2)-tq\geq \upkappa(p-2)-sp$, using Lemma~\ref{L3.6} and the above estimate we also have $J_2\geq -C_{\varepsilon_1} |\abar|^{\upkappa_1(p-1)-sp}$. 
Also,
\begin{align*}
J_3\geq - C|\abar|^\alpha \int_{\tilde\delta}^{\tilde\varrho} r^{\upkappa(q-1)-1-tq} dr &\geq - C|\abar|^\alpha \int_{\tilde\delta}^{1} r^{\upkappa(p-1)-1-sp} dr
\\
&\geq -\frac{C}{sp-\gamma(p-1)} \varepsilon_1^{\upkappa(p-1)-sp} |\abar|^{\alpha+ \upkappa(p-1)-sp}
\\
&\geq -\frac{C}{sp-\gamma(p-1)} \varepsilon_1^{\upkappa(p-1)-sp} |\abar|^{\upkappa_1(p-1)-sp},
\end{align*}
where the last line follows due to our choice of $\upkappa_1$ satisfying $\upkappa_1\leq \upkappa + \frac{\alpha}{p-1}$.
Because of the finiteness of $A$, it is evident that $|I_4| + |J_4|\leq C$.
Combing these estimates and \eqref{ET3.7A} in \eqref{E2.7} we obtain
$$C_1\geq |\abar|^{\upkappa_1(p-1)-sp} (\frac{C}{2}L^{p-1}-\tilde{C}_{\varepsilon_1})\geq |\abar|^{\gamma(p-1)-sp} (\frac{C}{2}L^{p-1}-\tilde{C}_{\varepsilon_1})$$
for some constant $\tilde{C}_{\varepsilon_1}$ (dependent on $\varepsilon_1$) and for all $L\geq L_0$. Since $\gamma(p-1)-sp<0$, $|\abar|\to 0$ as $L\to \infty$, this leads to a contradiction for a large enough $L$.
In other order, if we set a $L$ for which the above inequality fails to hold, we must have $\Phi$ in \eqref{E2.2} negative in $B_2\times B_2$. In particular,
$$u(x)-u(y)\leq L\varphi_{\upkappa_1}(|x-y|)= L|x-y|^{\upkappa_1}\quad \text{for all}\; x, y\in B_{\varrho_1}.$$
This proves that $u\in C^{0, \upkappa_1}(\bar{B}_{\varrho_1})$. Now we can follow a standard bootstrapping argument to conclude that $u\in C^{0, \gamma}_{\rm loc}(B_2)$.

\medskip

\noindent{\bf Step 2.}\;   We show that $u \in C^{0, \breve\gamma}_{\rm loc}(B_2)$.
First we suppose $\breve\gamma=\frac{sp}{p-1}<1$.
From Step 1 we know that $u\in C^{0, \gamma}_{\rm loc}(B_2)$ for any $\gamma<\breve\gamma$. We choose $\upkappa<\breve\gamma$ 
large enough so that $\upkappa(p-2) -sp>-1$ and $\breve\gamma-\upkappa<\frac{\alpha}{p-1}$. Due to this choice of $\upkappa$ we will have the following
\begin{gather*}
\upkappa(q-2)+2 -tq\geq \upkappa(p-2)+2 -sp>0, \upkappa(q-2)+1 -tq\geq \upkappa(p-2)+1 -sp>0.
\end{gather*}
We also set $\theta\in (0, 1)$ small enough so that $\upkappa+\theta[\upkappa(p-2)-sp]>0$. Since $u\in C^{0, \upkappa}(\bar{B}_{\varrho_2})$, from Lemma~\ref{L3.3}-\ref{L3.6} we obtain
$$I_3+I_4+J_2+J_4\geq -C_1, \quad J_3\geq -C_1 \varepsilon_1^{\upkappa(q-1)-tq}|\abar|^{\alpha+ \upkappa(p-1)-sp}\geq C_1 \varepsilon_1^{\upkappa(q-1)-tq}.$$
\eqref{ET3.7A} continuous to hold. Hence, from \eqref{E2.7}, we obtain
$$ \frac{C}{2}L^{p-1} - C_1 \varepsilon_1^{\upkappa(q-1)-tq}\leq C_1$$
for all $L\geq L_0$, where $\varepsilon_1$ is chosen small but fixed so that \eqref{ET3.7A} holds. But this is a contradiction for $L$ large enough. As argued before, it proves that
$u\in C^{0, \breve\gamma}(\bar{B}_{\varrho_1})$.

Now suppose  $\breve\gamma=1\leq \frac{sp}{p-1}$. We argue along the lines of \cite[Theorem~6.1]{BT25}. 
Applying Step 1, we can choose $\upkappa\in (0, 1)$ large enough such that $\upkappa(p-1) + 1 - sp>0$,
$\alpha + \upkappa(p-1)>(p-1)$ and $C^{0, \upkappa}(\bar{B}_{\varrho_2})$.
Now choose $m\geq 3$ large enough so that 
$$\frac{m-1}{m}\upkappa+ \upkappa(p-2) + 1 - sp>0.$$
The function $\varphi$ is chosen to be Lipschitz profile function as defined in Lemma~\ref{L3.1}.
We let 
$\delta_0=\delta_1(\log^2|\abar|)^{-1}$ from Lemma~\ref{L3.1}, $\delta =\varepsilon_1 (\log^{2\rho}(|\abar|))^{-1}|\abar|$ from Lemma~\ref{L3.2}
and $\tilde\delta=\varepsilon_1 (\log^{2\tilde\rho}(|\abar|))^{-1}|\abar|$ from Lemma~\ref{L3.5}. Then for a small but fixed $\varepsilon_1$ we get
$$I_1+I_2+J_1\geq \frac{C}{2} L^{p-1} |\abar|^{p-1-sp} (\log^2(|\abar|))^{-\uptheta}.$$
To compute a lower bound of $I_3$, we set $\theta\in (0, 1)$ small enough so that  $\upkappa + \theta[\kappa(p-2)  - sp]>0$.
Also, by our choice of $\upkappa$,
$$\int_\delta^{|\abar|^\theta} r^{\upkappa(p-2) + 1-sp} dr\geq C |\abar|^{\theta[\upkappa(p-2) + 2-sp]}.$$
Since $\frac{sp}{p-1}\geq 1$, we have $\upkappa(p-2)-sp\leq \upkappa(p-2) - (p-1)\leq (p-2)(\upkappa-1)-1<-1$. Thus
\begin{align*}
|\abar|^{\frac{m-1}{m}\upkappa} \int_\delta^{|\abar|^\theta} r^{\upkappa(p-2) -sp} dr\leq C_{\varepsilon_1} (\log^{2\rho}(|\abar|))^{sp-1-\upkappa(p-2)}
|\abar|^{\frac{m-1}{m}\upkappa+ \upkappa(p-2) + 1 - sp}.
\end{align*}
Combining these estimate in Lemma~\ref{L3.3} we find 
$I_3\geq -C_{\varepsilon_1} |\abar|^{\theta_\upkappa}$ for some positive $\theta_\upkappa$.
Using the property $p\leq q$ and $t q\leq sp$, we see that $J_2$ also satisfies a similar lower bound. Again, 
\begin{align*}
|\abar|^\alpha \int_{\tilde\delta}^{\tilde\varrho} r^{\upkappa(q-1)-1-tq} dr &\leq C \int_{\tilde\delta}^{\tilde\varrho} r^{\upkappa(p-1)-1-sp} dr
\\
&\leq C |\abar|^\alpha (\tilde\delta)^{\upkappa(p-1)-sp}
\\
&= C (\varepsilon_1\log^{2\tilde\rho}(|\abar|))^{sp-\upkappa(p-1)} |\abar|^{\alpha+\upkappa(p-1)-sp}
\\
&\leq C_{\varepsilon_1} |\abar|^{\theta_1}
\end{align*}
for some $\theta_1> (p-1)-sp$ (which is possible by our choice in $\upkappa$) and for $L\geq L_0$ (which makes $|\abar|$ small), where
$C_{\varepsilon_1}$ is a constant dependent on $\varepsilon_1$. By Lemma~\ref{L3.6} we then have $J_3\geq -C_{\varepsilon_1} |\abar|^{\theta_1}$.
Thus, using Lemma~\ref{L3.4} and gathering the above estimate in \eqref{E2.7} we arrive at
$$\frac{C}{2} L^{p-1} |\abar|^{p-1-sp} (\log^2(|\abar|))^{-\uptheta}\leq 
C_{\varepsilon_1} (|\abar|^{\theta_1}
+  |\abar|^{\theta_\upkappa}) + C \max\{|\abar|^\alpha, |\abar|^{\upkappa}, |\abar|^{\beta}\}.$$
Since $p-1-sp< \min\{\theta_1, \alpha, \beta, \upkappa\} $, the above inequality can not hold for small enough $|\abar|$, 
or equivalently for large enough $L$. This contradiction implies that $u\in C^{0, 1}_{\rm loc}(B_2)$.

\medskip

\noindent{\bf Step 3.}\; 
Let $\gamma<\gamma_\circ= \min\{1, \frac{sp+\alpha\wedge\beta}{p-1}, \frac{sp}{p-2}\}$.
We show that $u \in C^{0, \gamma}_{\textit{\rm loc}}(B_2)$.
In view Step 2, we only need to consider the situation $\breve\gamma<1$. Furthermore, we need to
consider $\gamma \in (\frac{sp}{p-1}, \gamma_\circ)$.
We employ a bootstrapping argument as in Step 1. To do so, let us assume that $u \in C^{0, \upkappa}(B_{\varrho_2})$
for some $\upkappa\in [\frac{sp}{p-1}, \gamma)$. Note that, due to Step 2, we can start the bootstrapping with
$\upkappa = \frac{sp}{p-1}$.  Also, we have $\upkappa(p-1)-sp\geq 0$. Now consider the parabola
\begin{equation}\label{AB04}
\ell(y)= y^2(p-2) -[sp+2(p-2)]y + 2sp.
\end{equation}
It is easy to check that $\ell(\frac{sp}{p-2})=0$, $\ell(1)=sp-(p-2)$ and $\ell'(1)= -sp<0$. Thus, for $\frac{sp}{p-2}<1$, the leftmost zero of $\ell$ is at $\frac{sp}{p-2}$ and
the function $\ell$ strictly decreases in $(-\infty, \frac{sp}{p-2}]$. When $\frac{sp}{p-2}\geq 1$, the function is strictly decreasing in $(-\infty, 1]$ with $\ell(1)\geq 0$. Thus, given 
$\gamma<\min\{1, \frac{sp}{p-2}\}$ we have
$$\uprho:=\frac{1}{4(p-1)}\min_{y\in [0, \gamma]} \ell(y)>0.$$
Now we set
  $$\upkappa_1=\min\{{\gamma}, \upkappa + \frac{1}{2(p-1)}, \upkappa + \frac{sp-(p-2)\gamma}{2(p-1)}, \upkappa + \frac{\alpha}{p-1}, \upkappa+\uprho\}.$$
 Also, set $m$ large enough so that $2\uprho-\frac{p(1-s)}{2m(p-1)}>\uprho$.
Using Lemma \ref{L3.1}, \ref{L3.2} and \ref{L3.5}, we can choose $\varepsilon_1$ small enough, independent of $L$, so that
\begin{equation}\label{ET3.7B}
I_1 + I_2 + J_1 \geq \frac{C}{2}L^{p-1} |\abar|^{\upkappa_1(p-1) - sp}.
\end{equation}
Define $\theta=\frac{\upkappa_1(p-1)-sp}{\upkappa(p-2)+2-sp}$. Since $\upkappa(p-2)+2-sp\geq 2-\upkappa>1$ and $\upkappa_1(p-1)-sp< (p-1)\upkappa-(p-2)\gamma\leq (p-1)\upkappa-(p-2)\upkappa=\upkappa$
we have $\theta\in (0, 1)$. We use this $\theta$ in Lemma~\ref{L3.3}. Therefore,
\begin{align*}
\int_{\delta}^{|\abar|^{\theta}} r^{\upkappa(p-2) +1-sp} &\leq \frac{1}{\upkappa(p-2)+2-sp} |\abar|^{\theta(\upkappa(p-2)-sp+2)}
\\
&\leq  |\abar|^{\upkappa_1(p-1)-sp}
\end{align*}
using the fact that $\upkappa(p-2)+2-sp>1$. Let us now show that
$$\frac{m-1}{m}\upkappa + \theta[\upkappa(p-2)-sp]> \upkappa_1 (p-1)-sp \Leftrightarrow \upkappa_1< \frac{1}{p-1}\left[\frac{(m-1)\upkappa}{2m}(\upkappa(p-2)+2-sp)+sp\right].$$
Note that
\begin{align*}
\frac{1}{p-1}\left[\frac{(m-1)\upkappa}{2m}(\upkappa(p-2)+2-sp)+sp\right]-\upkappa
&=\frac{1}{2(p-1)}\ell(\kappa) -\frac{1}{2m(p-1)} \upkappa(\upkappa(p-2)+2-sp)
\\
&\geq  2\uprho- \frac{p(1-s)}{2m(p-1)}>\uprho.
\end{align*}
Thus, by our choice of $\upkappa_1$, we get
$$\upkappa_1 (p-1)-sp< \frac{m-1}{m}\upkappa + \theta[\upkappa(p-2)-sp]< \upkappa + \theta[\upkappa(p-2)-sp].$$
Therefore, 
\begin{align*}
|\abar|^{\frac{m-1}{m}\upkappa } \int_{\delta}^{|\abar|^{\theta}} r^{\upkappa(p-2)-sp} &\leq \frac{1}{\upkappa(p-2)-sp+1} |\abar|^{\frac{m-1}{m} + \theta (\upkappa(p-2)-sp+1)}
\\
&\leq \frac{1}{1-\gamma} |\abar|^{\upkappa_1(p-1)-sp}
\end{align*}
using $\upkappa(p-2)-sp+1\geq 1-\upkappa\geq 1-\gamma$. Since $|\abar| < 1$, we also get
$$|\abar|^{\upkappa+\theta(\upkappa(p-2)-sp)} \leq |\abar|^{\upkappa_1(p-1) -sp}$$
Combining the above three estimates we have
$$I_3 \geq -C |\abar|^{\upkappa_1(p-1) -sp} $$
for some constant $C$. Since $q\geq p$ and $t q\leq sp$, the estimate of $J_2$ turns out to 
$$J_2 \geq -C |\abar|^{\upkappa_1(p-1) -sp} .$$
From Lemma \ref{L3.4}, we obtain
$$I_4\geq -C|\abar|^{\upkappa}\quad \text{and}\quad |J_4| \geq -C\max \{ |\abar|^{\alpha}, |\abar|^{\kappa} \}.$$
Since $\upkappa \geq \frac{sp}{p-1} \geq \frac{tq}{q-1}$ which implies $\upkappa(q-1) -tq \geq 0$. Therefore, from Lemma \ref{L3.6}, we have 
\begin{align*}
J_3 &\geq - C|\abar|^\alpha \int_{\tilde\delta}^{\tilde\varrho} r^{\upkappa(q-1)-1-tq} dr
\\ 
&\geq - C|\abar|^\alpha |\log(\tilde\delta)|=- C|\abar|^{\alpha} |\log(\varepsilon_1|\abar|)|.
\end{align*}
Combing these estimates and \eqref{ET3.7B} in \eqref{E2.7} we obtain
$$C_2(|\abar|^{\beta }+ |\abar|^{\upkappa }+|\abar|^{\alpha} |\log(\varepsilon_1|\abar|)|)\geq |\abar|^{\upkappa_1(p-1)-sp} \left(\frac{C}{2}L^{p-1}-C_1\right)$$
for some constants $C, C_1, C_2$ and for all $L\geq L_0$.
Since $\upkappa_1(p-1)-sp< \min \{ \alpha\wedge\beta, \upkappa \}$ and $|\abar|\to 0$ as $L\to \infty$, this leads to a contradiction for a large enough $L$.
In other order, if we set a $L$ for which the above inequality fails to hold, we must have $\Phi$ in \eqref{E2.2} negative in $B_2\times B_2$. In particular,
$$u(x)-u(y)\leq L\varphi_{\upkappa_1}(|x-y|)= L|x-y|^{\upkappa_1}\quad \text{for all}\; x, y\in B_{\varrho_1}.$$
This proves that $u\in C^{0, \upkappa_1}(\bar{B}_{\varrho_1})$. Now we can follow a standard bootstrapping argument to conclude that $u\in C^{0, \gamma}_{\rm loc}(B_2)$.

Suppose $\gamma_{\circ}=\gamma=\frac{sp+ \alpha \wedge \beta}{p-1} <\min \{ 1, \frac{sp}{p-2} \}$.
This implies $\gamma(p-1)-sp<\gamma$. 
Choose $\upkappa=\frac{sp+ \alpha \wedge \beta-\epsilon}{p-1} >\frac{tq}{q-1}$ for $\epsilon>0$ small such that
$\gamma(p-1)-sp<\upkappa<\gamma$. Then all other estimates above 
remain the same except $J_3$. By our choice of $\upkappa$, $\upkappa(q-1)-1-tq>-1$, which implies
$$J_3 \geq -C|\abar|^\alpha \int_{\tilde\delta}^{\tilde\varrho} r^{\upkappa(q-1)-1-tq} dr \geq -C|\abar|^\alpha. $$
Combining the estimates  in \eqref{E2.7} we obtain
$$C_2(|\abar|^{\beta }+ |\abar|^{\upkappa }+|\abar|^{\alpha} )\geq |\abar|^{\gamma(p-1)-sp} \left(\frac{C}{2}L^{p-1}-C_1\right).$$
Since $\alpha \wedge \beta =\gamma(p-1)-sp < \upkappa$
and $|\abar|\to 0$ as $L\to \infty$, this leads to a contradiction for a large enough $L$. Therefore, $u\in C^{0,\gamma}_{\rm loc}(B_2)$.

\medskip

\noindent{\bf Step 4.}\; Suppose $\min\{\frac{sp+\alpha\wedge\beta}{p-1}, \frac{sp}{p-2}\} >1$. We show that $u \in C^{0,1}(\bar{B}_1)$.
If $ \frac{sp}{p-1}\geq 1$ the proof follows from Step 2 . So we assume $\frac{sp}{p-1}<1$. Consider the doubling
function $\Phi$ with $\varphi$ being the Lipschitz profile function given by \eqref{varphi} and scaled as in Lemma~\ref{L3.1}.
Since $\frac{sp}{p-2} > 1$ we have $\ell(\frac{sp}{p-2})=0$, $\ell(1) =sp-(p-2) >0$ and $\ell'(1)=-sp<0$ (see \eqref{AB04}). 
So the function $\ell$ is strictly decreasing and positive in $(-\infty, 1]$.
Let 
$$\uprho =\inf_{\upkappa \in (0,1)}\frac{\ell(\upkappa)}{4(p-1)}=\frac{\ell(1)}{4(p-1)} > 0.$$
From Step 3 we know that $u \in C^{0, \upkappa}(B_{\frac{3}{2}})$ for all $\upkappa <1.$ For our purpose we fix any 
$\upkappa\in (\frac{sp}{p-1}, 1)$ such that $\alpha +\upkappa(p-1)> p-1$, $\upkappa+\uprho>1$ and $(p-1)-sp<\upkappa$. We repeat the scheme described in section~\ref{S-genstr} with $\varrho_2=\frac{3}{2}$
and $\varrho_1=1$. As before, we suppose
$\sup_{B_2\times B_2}\Phi>0$. 
We set the notation
\begin{gather*}
\delta_0=\eta_0=\varepsilon_1 (\log^2(|\abar|))^{-1}, \quad
\uptheta = \frac{n+1}{2}+p-sp, \quad
\tilde\rho=\frac{\frac{n+1}{2}+q-tq}{q-tq}, 
\\
\rho=\frac{\frac{n+1}{2}+p-sp}{p-sp},
\quad
\cone=\{z\in B_{\delta_0|\abar|} \; :\; |\langle \abar, z\rangle| \geq (1-\eta_0)|\abar||z|\},
\\
\cD_1=B_\delta\cap \cone^c\quad \cD_2= B_{\frac{1}{8}}\setminus(\cD_1\cup\cone).
\end{gather*}
 From Lemma~\ref{L3.1}(ii), Lemma~\ref{L3.2}(ii) and Lemma~\ref{L3.5}, by choosing $\varepsilon_1$ small enough independent of $L$ we get,
 $$I_1 +I_2 + J_1 \geq \frac{C}{2} L^{p-1} |\abar|^{p-1-sp} (\log^2(|\abar|))^{-\uptheta},$$
 for all $L\geq L_0$, for some constant $L_0$.

Define $\theta=\frac{(p-1)-sp + \epsilon}{\upkappa(p-2)+2-sp} $. Since $\upkappa(p-2)+2-sp\geq 2-\upkappa>1$ and $(p-1)-sp < (p-1)-(p-2)=1$
we have $\theta\in (0, 1)$ for $\epsilon$ small enough. We use this $\theta$ in Lemma~\ref{L3.3}.
It is then easy to see that
\begin{align*}
\int_{\delta}^{|\abar|^{\theta}} r^{\upkappa(p-2) +1-sp} &\leq \frac{1}{\upkappa(p-2)+2-sp} |\abar|^{\theta(\upkappa(p-2)-sp+2)}
\\
&\leq  |\abar|^{(p-1)-sp +  \epsilon}
\end{align*}
using the fact that $\upkappa(p-2)+2-sp>1$. Let us now verify that
$$
\upkappa + \theta[\upkappa(p-2)-sp]> (p-1)-sp 
\Leftrightarrow 
1 < \frac{1}{2(p-1)} \left[\upkappa(\upkappa(p-2)+2-sp)+2sp \right] +\tilde\epsilon,
$$
where $\tilde\epsilon=\frac{\epsilon}{2(p-1)}$.
Note that
\begin{align*}
\frac{1}{2(p-1)} \left[\upkappa(\upkappa(p-2)+2-sp)+2sp \right]-\upkappa
&=  \frac{\ell(\upkappa)}{2(p-1)})
\\
&\geq  \uprho.
\end{align*}
Since $\upkappa+\uprho>1$, for $\epsilon$ small enough we get
$$(p-1)-sp<  \upkappa + \theta[\upkappa(p-2)-sp].$$
Since $|\abar| < 1$, we  have
$$|\abar|^{\upkappa+\theta(\upkappa(p-2)-sp)} \leq |\abar|^{(p-1) -sp +\varepsilon_2}$$
for some $\varepsilon_2>0$.
Set $m$ large enough such that $\frac{\upkappa}{m} < \theta$ which would imply $\upkappa\frac{m-1}{m} + \theta(\upkappa(p-2)-sp+1) > \upkappa + \theta(\upkappa(p-2)-sp)$.
Therefore, for some $\varepsilon_3>0$ we have
\begin{align*}
|\abar|^{\frac{m-1}{m}\upkappa } \int_{\delta}^{|\abar|^{\theta}} r^{\upkappa(p-2)-sp} &\leq \frac{1}{\upkappa(p-2)-sp+1} |\abar|^{\upkappa \frac{m-1}{m} + \theta (\upkappa(p-2)-sp+1)}
\\
&\leq C |\abar|^{(p-1)-sp + \varepsilon_3},
\end{align*}
Combining the above three estimates we have
$$I_3 \geq -C |\abar|^{(p-1) -sp +\varepsilon_4} $$
for some constants $C$ and $\varepsilon_4$. Since $q\geq p$ and $t q\leq sp$ the estimate of $J_2$ also turns out to be
$$J_2 \geq -C |\abar|^{(p-1) -sp +\varepsilon_4}.$$
From Lemma \ref{L3.4} we also have
$$I_4 \geq -C |\abar|^{\upkappa}$$
and
$$J_4 \geq -C ( |\abar|^{\upkappa} + |\abar|^{\alpha}).$$
Again, since $\upkappa(p-1)>sp$, we have
\begin{align*}
|\abar|^\alpha \int_{\tilde\delta}^{\tilde\varrho} r^{\upkappa(q-1)-1-tq} dr &\leq C \int_{\tilde\delta}^{\tilde\varrho} r^{\upkappa(p-1)-1-sp} dr
\\
&\leq \frac{C}{\upkappa(p-1)-sp} |\abar|^\alpha (\tilde\varrho)^{\upkappa(p-1)-sp}
\\
&= C_1  |\abar|^{\alpha}
\end{align*}
for some constant $C_1$.
By Lemma \ref{L3.6} we then have $J_3 \geq -C_1|\abar|^{\alpha}$.
Now putting the above estimates in \eqref{E2.7} we get
$$\tilde{C}(|\abar|^{\upkappa} + |\abar|^{\alpha} + |\abar|^{\beta} + |\abar|^{(p-1)-sp +\varepsilon_4}) \geq \frac{C}{2} L^{p-1} |\abar|^{p-1-sp} (\log^2(|\abar|))^{-\uptheta}.$$
Since $(p-1)-sp< \min \{ \alpha, \beta, \upkappa \}$, $|\abar|\to 0$ as $L\to \infty$, this leads to a contradiction in the same way as before for a large enough $L$. Thus, the proof is completed.
\end{proof}

\begin{rem}
In Step 3 above, if we set $\alpha=\beta=1$ (this includes the fractional $p$-harmonic case), the maximum possible value of $\upkappa_1$ in each iteration would be $\frac{1}{p-1}(\upkappa+sp)$.
Therefore, if we start with $\upkappa=\frac{sp}{p-1}$, after $k$-th iterations, we gain regularity 
$$
sp\left(\frac{1}{p-1}+\frac{1}{(p-1)^2}+\ldots + \frac{1}{(p-1)^{k+1}}\right).
$$
Letting $k\to\infty$, we get the limiting regularity $\frac{sp}{p-2}$. This is the reason we do not achieve the exponent $\frac{sp}{p-2}$
with the  iteration scheme.
\end{rem}

\section{Proof in the subquadratic case $1<p\leq 2$}\label{s-sub}
For the proof in this subquadratic case, we follow our general strategy stated in Section~\ref{S-genstr}. As done in Section~\ref{s-sup},
we begin with the estimates of $I_i, i=1,.., 4$ and $J_i, i=1,..,4$. Note that we already have established appropriate estimates for 
$I_1, I_2, I_4, J_1, J_4$ in the previous section. So we begin with the estimate on $I_3$.

\begin{lem}[Estimate of $I_3$]\label{L4.1}
Let $p\in (1,2]$. Suppose that $u\in C^{0, \upkappa}(\bar{B}_{\varrho_2})$ for some $\upkappa\in [0,  1)$. 
Let $\delta=\delta_0|\abar|$ where $\delta_0$ is chosen as in Lemma~\ref{L3.1}.
Then for any $\theta\in (0, 1)$,  we have $L_0>0$ such that
$$I_3\geq -C \left[\int_\delta^{|\abar|^\theta} r^{2(p-1) - 1-sp} dr  + 
|\abar|^{\frac{m-1}{m}\upkappa(p-1)} \int_\delta^{|\abar|^\theta} r^{p-2 -sp} dr + |\abar|^{\upkappa(p-1) -\theta sp}\right]$$
for all $L\geq L_0$, where the constant $C$ depends on $\upkappa, p, s, n, m, m_1$ and the $C^{0, \upkappa}$ norm of $u$ in
$B_{\varrho_2}$.
\end{lem}

\begin{proof}
We compute $I_3$ as in Lemma~\ref{L3.3}. We denote $\hat\delta=|\abar|^\theta$ and write
$$
I_3=\underbrace{\sL_p[\cD_2\cap B_{\hat\delta}] w_1(\xbar)-\sL_p[\cD_2\cap B_{\hat\delta}] w_2(\ybar)}_{I_{1, 3}} +
\underbrace{\sL_p[\cD_2\cap B^c_{\hat\delta}] w_1(\xbar)-\sL_p[\cD_2\cap B^c_{\hat\delta}] w_2(\ybar)}_{I_{2, 3}}.
$$
Since $|J_p(a)-J_p(b)|\leq 2|a-b|^{p-1}$, we get
\begin{align*}
|I_{2,3}| &\leq \int_{\cD_2 \cap B_{\hat\delta}^c}|J_p(\delta^1u(\xbar, z))-J_p(\delta^1u(\xbar, z))|
\\
&\leq C|\abar|^{\kappa(p-1)} \int_{\cD_2 \cap B_{\hat\delta}^c} |z|^{-n-sp}
\\
&\leq C |\abar|^{\kappa(p-1)-\theta sp}
\end{align*}
for some constant $C$, dependent on $\norm{u}_{C^{0,\upkappa}(B_{\varrho_2})}$. Again, since
$$\triangle u(\xbar, z) - \triangle u(\ybar, z)\geq m_1 \triangle \psi(\xbar, z),$$
using monotonicity of $J_p$ we obtain
\begin{align*}
I_{1, 3} &\geq \int_{\cD_2\cap B_{\hat\delta}} J_p(\triangle u(\ybar,z)+m_1\triangle\psi(\xbar,z))\frac{dz}{|z|^{n+sp}} -
 \int_{\cD_2\cap B_{\hat\delta}} J_p(\triangle u(\ybar,z))\frac{dz}{|z|^{n+sp}}
 \\
 &\geq -2 m_1^{p-1} \int_{B^c_\delta\cap B_{\hat\delta}} |\triangle\psi(\xbar,z)|^{p-1}\frac{dz}{|z|^{n+sp}}.
\end{align*}
Using $|\triangle \psi(\xbar, z)|\leq \kappa (|z|^2 + |\abar|^{\frac{m-1}{m}\upkappa}|z|)$ (see Lemma~\ref{L3.3}), the above 
computation leads to
$$I_{1, 3}\geq -C \left[ \int_{\delta}^{\hat\delta} r^{2(p-1)-1-sp} dz+ |\abar|^{\frac{m-1}{m}(p-1)\upkappa}
\int_{\delta}^{\hat\delta} r^{p-2-sp} \dz \right].
$$
This completes the proof.
\end{proof}

Next we estimate $J_2$ and $J_3$.
\begin{lem}\label{L4.2}
Let $p\in (1,2]$ and $q\in[p, \infty)$. Also, assume that $u\in C^{0, \upkappa}(B_{\varrho_2})$ for some $\upkappa\in [0, 1)$. 
Let $\tilde\delta$ be given by Lemma~\ref{L3.5}. Then for any $\theta\in (0, 1)$,  we have $L_0>0$ such that
\begin{equation}\label{EL4.2A}
J_2\geq \left\{
\begin{split}
& -C \left[\int_{\tilde\delta}^{|\abar|^\theta} r^{2(p-1) - 1-sp} dr  + 
|\abar|^{\frac{m-1}{m}\upkappa(p-1)} \int_{\tilde\delta}^{|\abar|^\theta} r^{p-2 -sp} dr + |\abar|^{\upkappa(p-1) -\theta sp}\right]\; \text{for}\; q\leq 2,
\\
&-C \left[\int_\delta^{|\abar|^\theta} r^{\upkappa(q-2) + 1- tq} dr  + 
|\abar|^{\frac{m-1}{m}\upkappa} \int_\delta^{|\abar|^\theta} r^{\upkappa(q-2) -tq} dr + |\abar|^{\upkappa +\theta(\upkappa(q-2)-tq)}\right]\;
\text{for}\; q>2,
\end{split}
\right.
\end{equation}
for all $L\geq L_0$, where the constant $C$ depends on $\upkappa, M, p, s, n, m, m_1$ and the $C^{0, \upkappa}$ norm of $u$ in
$B_{\varrho_2}$. Moreover, if $\xi$ is $\alpha$-H\"{o}lder continuous in the first argument uniformly with respect to the second, then
$$J_3\geq - C|\abar|^\alpha \int_{\tilde\delta}^{\tilde\varrho} r^{\upkappa(q-1)-1-tq} dr,$$
where the constant $C$ depends on $\sup_z\norm{\xi(\cdot, z)}_{C^{0, \alpha}}$ and the $C^{0, \upkappa}$ norm of $u$ in
$B_{\varrho_2}$.
\end{lem}

\begin{proof}
The first inequality in \eqref{EL4.2A} follows from the proof of Lemma~\ref{L4.1} combined with the fact that $p\leq q$ and $tq\leq sp$,
whereas the second one follows from Lemma~\ref{L3.6}.
Estimate $J_3$ follows from Lemma~\ref{L3.6}.
\end{proof}

Now we can state our main result of this section.

\begin{thm}\label{T4.3}
Let $p\in (1,2]$, $p\leq q$ and $\xi$ be $\alpha$-H\"{o}lder continuous. Also, let $f\in C^{0,\beta}(\bar{B}_2)$
and $u\in C(B_2)\cap L^{p-1}_{sp}(\Rn)\cap L^{q-1}_{tq}(\Rn)$ be a viscosity solution to \eqref{E2.1}. Then $u\in C^{0, \gamma}(\bar{B}_1)$ for any $\gamma<\gamma_\circ$ where
$\gamma_\circ=\min\{\frac{sp+\alpha\wedge\beta}{p-1}, 1\}$. Furthermore, when $\frac{sp+\alpha\wedge\beta}{p-1}<1$, we have 
$u\in C^{0, \gamma_\circ}(\bar{B}_1)$, and
when $\frac{sp+\alpha\wedge\beta}{p-1}>1$, we have 
$u\in C^{0, 1}(\bar{B}_1)$.
The H\"{o}lder norm $\norm{u}_{C^{0, \gamma}(B_1)}$, including $\gamma=1$, is bounded above  by a constant dependent on the data $M, A, s, p, t, q, n$.
\end{thm}

\begin{proof}
We broadly follow the strategy of Theorem~\ref{T3.7} and divide the proof into multiple steps as follows:
\begin{itemize}
\item[Step 1] --We show $u\in C^{0, \gamma}_{\rm loc}(B_2)$ for any $\gamma<\min\{1, \frac{sp}{p-1}\}=\breve{\gamma}$.
\item[Step 2]-- We show that $u \in C^{0, \breve\gamma}_{\rm loc}(B_2)$.
\item[Step 3]-- We show that $u \in C^{0, \gamma}_{\textit{\rm loc}}(B_2)$ for any $\gamma<\gamma_\circ$ and
 when $\gamma_{\circ}=\frac{sp+ \alpha \wedge \beta}{p-1} < 1$, we have $u \in C^{0, \gamma_\circ}_{\textit{\rm loc}}(B_2)$.
\item[Step 4]-- Assuming $\frac{sp+ \alpha \wedge \beta}{p-1} >1$, we show that $u \in C^{0,1}_{\rm loc}(B_2)$.
\end{itemize}
\medskip

\noindent{\bf Step 1.} Let $\breve\gamma=\min\{1, \frac{sp}{p-1}\}$. We show that $u\in C^{0, \gamma}_{\rm loc}(B_2)$ for any $\gamma<\breve\gamma$.
We apply a bootstrapping argument. Fix $1\leq\varrho_1<\varrho_2< 2$ and also assume $u\in C^{0, \upkappa}(B_{\varrho_2})$ for some
$\upkappa\in [0, \gamma)$. 
Set
$$\upkappa_1=\min\{\gamma, \upkappa + \frac{sp}{2(p-1)}, \upkappa + \frac{1}{2(q-1)}, \upkappa + \frac{\alpha}{p-1}\}.$$
With $\varphi=\varphi_{\upkappa_1}$ from \eqref{varphi} we consider the doubling function $\Phi$ in \eqref{E2.2}.  
We set $\delta=\tilde\delta=\varepsilon_1|\abar|$. Using Lemma~\ref{L3.1}, ~\ref{L3.2} and ~\ref{L3.5}, we can choose
$\varepsilon_1$ small enough, independent of $L$, so that
\begin{equation}\label{ET4.3A}
I_1+I_2+J_1\geq \frac{C}{2} L^{p-1} |\abar|^{\upkappa_1(p-1)-sp}
\end{equation}
for all $L\geq L_0$. We choose $\theta\in (0, 1)$ small enough and $m$ large enough so that $\frac{m-1}{m}\upkappa(p-1)-\theta sp> \upkappa_1(p-1)-sp$. Then
 from Lemma~\ref{L4.1} we have
\begin{align*}
&\int_\delta^{|\abar|^\theta} r^{2(p-1) - 1-sp} dr  + 
|\abar|^{\frac{m-1}{m}\upkappa(p-1)} \int_\delta^{|\abar|^\theta} r^{p-2 -sp} dr + |\abar|^{\upkappa(p-1) -\theta sp}
\\
& \leq \int_\delta^1  r^{\upkappa_1(p-1)-sp-1} dr 
+|\abar|^{\frac{m-1}{m}\upkappa(p-1)} \int_\delta^{|\abar|^\theta} r^{-1 -sp} dr
+ |\abar|^{\upkappa_1(p-1) - sp}
\\
&\leq C \varepsilon_1^{-sp} 
|\abar|^{\upkappa_1(p-1) - sp}.
\end{align*}
Thus we get
$$I_3\geq - C \varepsilon_1^{-sp} |\abar|^{\upkappa_1(p-1) - sp}.$$
Now if $q\leq 2$, then using the bound in Lemma~\ref{L4.2} and the fact $p\leq q, tq\leq sp$ we would get
$$J_2\geq - C \varepsilon_1^{\upkappa_1(q-1)-tq} |\abar|^{\upkappa_1(p-1) - sp}.$$
Let us now assume that $q>2$. In this case, we also
choose $m$ large enough so that
$$\upkappa(q-2) + 1>\upkappa_1(q-1)-\frac{m-1}{m}\upkappa.$$
Now, we set $\theta$ small enough so that $\upkappa +\theta(\upkappa(q-2)-tq)>\upkappa_1(p-1)-sp$. This gives us
$$|\abar|^{\upkappa +\theta(\upkappa(q-2)-tq)}\leq |\abar|^{\upkappa_1(p-1)-sp}.$$
Now repeating the calculation in Step 1 of Theorem~\ref{T3.7}, with $(q, t)$ in place of $(p, s)$, gives
\begin{align*}
|\abar|^{\frac{m}{m-1}\upkappa}\int_{\delta}^{|\abar|^\theta} r^{\upkappa (q-2)-tq}dr\leq 
& C \varepsilon_1^{\upkappa_1 (q-1)-\frac{m}{m-1}\upkappa -t q} |\abar|^{\upkappa_1 (q-1)-tq}
\leq C \varepsilon_1^{\upkappa_1 (q-1)-\frac{m}{m-1}\upkappa -t q} |\abar|^{\upkappa_1 (p-1)-sp}
\\
\int_\delta^{|\abar|^\theta} r^{\upkappa(q-2) + 1-tq} dr &\leq C \varepsilon_1^{\upkappa_1(q-1)-tq}|\abar|^{\upkappa_1(q-1)-tq}
\leq C \varepsilon_1^{\upkappa_1(q-1)-tq} |\abar|^{\upkappa_1(p-1)-sp}
\end{align*}
for some constant $C$, dependent on $m, \upkappa, t, q$. Hence, for some constant $C_{\varepsilon_1}$ we would have
$$J_2\geq -C_{\varepsilon_1}|\abar|^{\upkappa_1(p-1)-sp}.$$
Again,
\begin{align*}
J_3 &\geq - C|\abar|^\alpha \int_{\tilde\delta}^{\tilde\varrho} r^{\upkappa(q-1)-1-tq} dr
\\
&\geq - C|\abar|^\alpha \int_{\tilde\delta}^{\tilde\varrho} r^{\upkappa(p-1)-1-sp} dr 
\\
&\geq  -C \varepsilon_1^{\upkappa(p-1)-sp} |\abar|^{\upkappa_1(p-1)-sp}.
\end{align*}
Now since $|I_4|+|J_4|\leq C$, by Lemma~\ref{L3.4}, plugging in these estimates in \eqref{E2.7} we obtain
$$C\geq |\abar|^{\upkappa_1(p-1)-sp} (\frac{C}{2}L^{p-1}-\tilde{C}_{\varepsilon_1})\geq |\abar|^{\gamma(p-1)-sp} (\frac{C}{2}L^{p-1}-\tilde{C}_{\varepsilon_1})$$
for some constant $\tilde{C}_{\varepsilon_1}$ (dependent on $\varepsilon_1$) and for all $L\geq L_0$. This is a contradiction and rest of the argument 
follows as before. This proves that $u\in C^{0, \upkappa_1}(\bar{B}_{\varrho_1})$ and by a standard bootstrapping argument we would have
$u\in C^{0,\gamma}_{\rm loc}(B_2)$ for any $\gamma<\breve\gamma$.

\medskip

\noindent{\bf Step 2.} We prove that $u\in C^{0, \breve\gamma}_{\rm loc}(B_2)$. From Step 1 we know that $u\in C^{0, \gamma}_{\rm loc}(B_2)$ for any $\gamma<\breve\gamma$.
First we suppose $\breve\gamma=\frac{sp}{p-1}<1$, implying $\frac{tq}{q-1}<1$. For $q\leq 2$, the proof follows from Step 1. So we suppose that $q>2$.
 We choose $\upkappa<\breve\gamma$ 
large enough so that $\upkappa(q-2) -tq>-1$ and $\breve\gamma-\upkappa<\frac{\alpha}{q-1}$. 
Note that \eqref{ET4.3A} holds with $\breve\gamma$ in place of $\upkappa_1$. Estimates of $I_3, I_4, J_3, I_4$ remain same as in Step 1. Because 
of our choice in $\upkappa$ we would also have $J_2\geq -C$. Now the proof can be completed as done before.

Next we suppose $\breve\gamma=1$, equivalently, $\frac{sp}{p-1}\geq 1$. Applying Step 1, we can choose $\upkappa\in (\frac{p-1}{p}, 1)$ large enough such that $\upkappa(q-1) + 1 - tq>0$,
$\alpha + \upkappa(q-1)>(q-1)$, $\alpha + \upkappa(p-1)>(p-1)$ and $C^{0, \upkappa}(\bar{B}_{\varrho_2})$.
Now choose $m\geq 3$ large enough so that 
$$\frac{m-1}{m}\upkappa+ \upkappa(q-2) + 1 - tq>0\quad \text{and}\quad \frac{m-1}{m}\upkappa + (p-1)\upkappa>(p-1).$$
The function $\varphi$ is chosen to be Lipschitz profile function as defined in Lemma~\ref{L3.1}.
We let 
$\delta_0=\delta_1(\log^2|\abar|)^{-1}$ from Lemma~\ref{L3.1}, $\delta =\varepsilon_1 (\log^{2\rho}(|\abar|))^{-1}|\abar|$ from Lemma~\ref{L3.2}
and $\tilde\delta=\varepsilon_1 (\log^{2\tilde\rho}(|\abar|))^{-1}|\abar|$ from Lemma~\ref{L3.5}. Then for a small but fixed $\varepsilon_1$ we get
$$I_1+I_2+J_1\geq \frac{C}{2} L^{p-1} |\abar|^{p-1-sp} (\log^2(|\abar|))^{-\uptheta}$$
for all $L\geq L_0$. Since $(p-1)< (p-1)\upkappa + sp$, we choose $\theta\in (0, 1)$ small enough so that $(p-1-sp)<\upkappa(p-1)- \theta sp$. 
Now, we see that 
$$|\abar|^{\frac{m-1}{m}\upkappa(p-1)} \int_\delta^{|\abar|^{\theta}} r^{p-2-sp} dr \leq C |\abar|^{\frac{m-1}{m} \upkappa(p-1)} [ |\log(\delta)|1_{\{sp=p-1\}} + \delta^{p-1-sp}1_{\{sp>p-1\}}].$$
Form the definition of $\delta$ and the fact that $|\abar|\to 0$ as $L\to \infty$, we can find $L_0, \eta_1>0$ so that 
$$|\abar|^{\frac{m-1}{m}\upkappa(p-1)} \int_\delta^{|\abar|^{\theta}} r^{p-2-sp} dr\leq C_{\varepsilon_1} |\abar|^{p-1-sp+\eta_1}$$
for $L\geq L_0$, where the constant $C_{\varepsilon_1}$ depends on $\varepsilon_1$. On the other hand, if $2(p-1)-sp>0$ we get
$$\int_\delta^{|\abar|^\theta} r^{2(p-1) - 1-sp} dr\leq C |\abar|^{\theta(2(p-1)-sp)},$$
and if $2(p-1)-sp\leq 0$, we have $(p-1)-sp + \frac{p-1}{2}< 0$, giving us
$$\int_\delta^{|\abar|^\theta} r^{2(p-1) - 1-sp} dr\leq \int_\delta^{1} r^{(p-1)+\frac{p-1}{2} - 1-sp} dr\leq C_{\varepsilon_1}|\abar|^{p-1-sp +\eta_2}$$
for some positive $\eta_2$ and $L\geq L_0$, provided we choose $L_0$ large enough. Thus, combining these estimates
in Lemma~\ref{L4.1}, we find $\eta>0$ such that
$$I_3\geq -C_{\varepsilon_1}|\abar|^{p-1-sp +\eta}.$$
Note that, by Lemma~\ref{L4.2}, $J_2$ also satisfies a similar estimate if $q\leq 2$. For $q>2$ the estimates in Step 2 of Theorem~\ref{T3.7} goes through, giving us for some $\eta_3>0$ that
$$J_2\geq -C |\abar|^{\eta_3}\geq - C|\abar|^{\eta_3 + p-1-sp}.$$
Also,
\begin{align*}
J_3 &\geq - C|\abar|^\alpha \int_{\tilde\delta}^{\tilde\varrho} r^{\upkappa(q-1)-1-tq} dr
\\
&\geq - C|\abar|^\alpha \int_{\tilde\delta}^{\tilde\varrho} r^{\upkappa(p-1)-1-sp} dr \geq - C_{\varepsilon_1} |\abar|^{p-1-sp-\eta_4}
\end{align*}
for some $\eta_4>0$ and $L\geq L_0$ for some large $L_0$, where we use the fact $\alpha+\upkappa(p-1)>p-1$. Thus, combining this estimates and using Lemma~\ref{L3.4} in \eqref{E2.7} gives us
$$\frac{C}{2} L^{p-1} |\abar|^{p-1-sp} (\log^2(|\abar|))^{-\uptheta}\leq C_{\varepsilon_1}(|\abar|^{p-1-sp +\eta\wedge\eta_3\wedge\eta_4}, |\abar|^{\alpha\wedge\beta}, |\abar|^{\upkappa(p-1)})$$
for all $L\geq L_0$, where $L_0$ is some large number chosen according to the above estimates. Since $p-1-sp\leq 0$ and $|\abar|\to 0$ as $L\to\infty$, the above estimate can not hold for large enough $L$.
Thus contradiction proves that $u\in C^{0, 1}(\bar{B}_{\varrho_1})$.

\medskip

\noindent{\bf Step 3.} Fix $\gamma<\gamma_\circ$. We prove that $u\in C^{0, \gamma}_{\rm loc}(B_2)$. In view of Step 2 above, we assume that $\gamma\in (\breve\gamma, \gamma_\circ)$
and $\breve\gamma<1$. We apply a bootstrapping argument similar to Step 3 of Theorem~\ref{T3.7}. Note that, due to Step 2, we can start the bootstrapping with
$\upkappa = \frac{sp}{p-1}$.  Also, we have $\upkappa(p-1)-sp\geq 0$. We need the following parabola (see \eqref{AB04}) for $q>2$
$$\ell_q(y)=y^2(q-2)-[tq+2(p-2)]y +2sp$$
Since $sp\geq tq$ and $p \leq 2$ we have
$$\ell_q(y)=y^2(q-2)+2y(2-p)+[2sp-y\, tq]>0$$
whenever $y\in [ 0, 2)$. In this case, we define
$$\uprho=\frac{1}{4(p-1)}\inf_{y \in (0,\gamma)}\ell_q(z).$$
Now let $\upkappa_1=\min \{\gamma, \upkappa + \frac{sp(1-\gamma)}{2(p-1)-sp}, \upkappa+ \frac{sp}{2} , \upkappa +\frac{\alpha}{2(q-1)},  \upkappa +\uprho \}$.
Choose $m$ large enough such that $\upkappa_1(p-1)-sp < \frac{m-1}{m}\upkappa(p-1)$ and $\frac{q}{2m(p-1)}<\uprho$.
We set $\delta=\tilde\delta=\varepsilon_1|\abar|$. Using Lemma~\ref{L3.1}, ~\ref{L3.2} and ~\ref{L3.5}, we can choose
$\varepsilon_1$ small enough, independent of $L$, so that
\begin{equation}\label{ET4.3B}
I_1+I_2+J_1\geq \frac{C}{2} L^{p-1} |\abar|^{\upkappa_1(p-1)-sp}
\end{equation}
for all $L\geq L_0$. Since $sp<p-1$, from Lemma~\ref{L4.1} we obtain
$$I_3\geq -C \left[|\abar|^{\theta(2(p-1)-sp)} + |\abar|^{\frac{m-1}{m}\upkappa(p-1) + \theta((p-1)-sp)} + |\abar|^{\upkappa(p-1)-\theta sp}\right]$$
for $\theta\in (0, 1)$. Now we set $\theta=\frac{\upkappa_1(p-1)-sp}{2(p-1)-sp}\leq \frac{\upkappa(p-1)}{2(p-1)-sp}=\frac{\upkappa}{2-\nicefrac{sp}{p-1}}<1$.
Again, since
$$\upkappa_1 \leq \upkappa + sp \frac{1-\gamma}{2(p-1)-sp}< \upkappa + \frac{sp}{p-1}\frac{2(p-1)-\upkappa_1(p-1)}{2(p-1)-sp}=\upkappa + \frac{sp}{p-1}(1-\theta),$$
we get $\upkappa_1(p-1)-sp< \upkappa(p-1) -\theta sp$. Also, by our choice of $m$,
$$(\upkappa_1(p-1)-sp)-\theta (p-1-sp)=(\upkappa_1(p-1)-sp)\frac{p-1}{2(p-1)-sp}< (\upkappa_1(p-1)-sp)< \frac{m-1}{m}\upkappa(p-1),$$
implying
$$ \frac{m-1}{m}\upkappa(p-1)+ \theta ((p-1)-sp)> (\upkappa_1(p-1)-sp).$$
These inequalities lead to $I_3\geq -C |\abar|^{\upkappa_1(p-1)-sp}$. When $q\leq 2$, using \eqref{EL4.2A}, we see that $J_2$ also has a similar lower bound.
So we estimate $J_2$ when $q>2$. We set $\theta = \frac{\upkappa_1(p-1)-sp}{\upkappa(q-2)-tq+2}$. Since $2> sp\geq tq$ and
$\upkappa_1> \upkappa\geq \frac{sp}{p-1}$, we have $\theta>0$. Moreover, $\upkappa_1(p-1)-sp \leq  \upkappa_1(q-1)-tq
 < \upkappa(q-2)-tq +2$. Thus, $\theta\in (0, 1)$. 
From the definition of $\uprho$ we have
\begin{align*}
\frac{1}{p-1}\left[\frac{(m-1)\upkappa}{2m}(\upkappa(q-2)+2-tq)+sp\right]-\upkappa
&=\frac{1}{2(p-1)}\ell_q(\upkappa) -\frac{1}{2m(p-1)} \upkappa(\upkappa(q-2)+2-tq)
\\
&\geq  2\uprho- \frac{q}{2m(p-1)}>\uprho,
\end{align*}
giving us
$$\upkappa_1< \frac{1}{p-1}\left[\frac{(m-1)\upkappa}{2m}(\upkappa(q-2)+2-tq)+sp\right]\Leftrightarrow
\frac{m-1}{m}\upkappa + \theta[\upkappa(q-2)-tq]> \upkappa_1 (p-1)-sp.$$
Since $\upkappa\geq \frac{sp}{p-1}\geq \frac{tq}{q-1}\Rightarrow \upkappa(q-2) - tq\geq -\upkappa>-1$,  we obtain
\begin{align*}
|\abar|^{\frac{m-1}{m}\upkappa} \int_\delta^{|\abar|^\theta} r^{\upkappa(q-2) -tq} dr\leq C |\abar|^{\frac{m-1}{m}\upkappa+\theta(\upkappa(q-2) -tq+1)}
\leq C |\abar|^{\upkappa_1(p-1)-sp},
\end{align*}
and
$$
|\abar|^{\upkappa +\theta(\upkappa(q-2)-tq)}\leq |\abar|^{\upkappa_1(p-1)-sp}.
$$
Thus, we have $J_2\geq -C |\abar|^{\upkappa_1(p-1)-sp}$. The term $J_3$ can be estimated as follows
\begin{align*}
J_3 &\geq - C|\abar|^\alpha \int_{\tilde\delta}^{\tilde\varrho} r^{\upkappa(q-1)-1-tq} dr
\\
&\geq -C|\abar|^\alpha (|\abar|^{\upkappa(q-1)-tq}1_{\{\upkappa(q-1)-tq>0\}} + |\log(\tilde\delta)|1_{\{\upkappa(q-1)-tq=0\}})
\\
&\geq -C_{\varepsilon_1} |\abar|^{\frac{\alpha}{2}+\upkappa(q-1)-tq}\geq -C_{\varepsilon_1} |\abar|^{\upkappa_1(q-1)-tq}
\geq -C_{\varepsilon_1} |\abar|^{\upkappa_1(p-1)-sp}
\end{align*}
for all $L\geq L_0$, where $C_{\varepsilon_1}$ depends on $\varepsilon_1$.
Now combining \eqref{ET4.3B} with the above estimates and Lemma~\ref{L3.4} in \eqref{E2.7} gives us
$$\left(\frac{C}{2} L^{p-1} -C_{\varepsilon_1}\right)|\abar|^{\upkappa_1(p-1)-sp}\leq C \max\{|\abar|^{\alpha\wedge\beta}, |\abar|^{\upkappa(p-1)}\}$$
for all $L\geq L_0$. Since by our assumption $\upkappa_1(p-1)-sp< \min\{\alpha\wedge\beta, \upkappa(p-1)\}$, the above inequality can not hold for 
large $L$. Hence we have a contradiction and arguing as before we get $u\in C^{0, \upkappa_1}(\bar{B}_{\varrho_1})$. By a standard bootstrapping
argument we then get $u\in C^{0, \gamma}_{\rm loc}(B_2)$.

Suppose $\gamma_{\circ}=\gamma=\frac{sp+ \alpha \wedge \beta}{p-1} < 1$.
Choose $\upkappa=\frac{sp+ \alpha \wedge \beta-\epsilon}{p-1} $ for $\epsilon \in (0, sp)$ small. Then 
$\upkappa(p-1) =\alpha \wedge \beta +sp-\epsilon >\alpha \wedge \beta$ by our choice of $\epsilon$.
We can repeat the above calculations with $\upkappa_1=\gamma$ leading to, after gathering them in \eqref{E2.7},
$$\left(\frac{C}{2} L^{p-1} -C_{\varepsilon_1}\right)|\abar|^{\gamma(p-1)-sp}\leq C \max\{|\abar|^{\alpha\wedge\beta}, |\abar|^{\upkappa(p-1)}\}$$
Since $\alpha \wedge \beta =\gamma(p-1)-sp < \upkappa(p-1)$
and $|\abar|\to 0$ as $L\to \infty$, this leads to a contradiction for a large enough $L$. Therefore, $u\in C^{0,\gamma}_{\rm loc}(B_2)$.

\medskip

\noindent{\bf Step 4.} Let $\frac{sp+\alpha\wedge\beta}{p-1}>1$. We show that $u\in C^{0, 1}(\bar{B}_1)$. 
If $ \frac{sp}{p-1}\geq 1$ the proof follows from Step 2 . So we assume $\frac{sp}{p-1}<1$. Consider the doubling
function $\Phi$ with $\varphi$ being the Lipschitz profile function given by \eqref{varphi} and in Lemma~\ref{L3.1}. From Step 3 we know that
$u\in C^{0, \upkappa}_{\rm loc}(B_2)$ for all $\upkappa\in (0, 1)$. For our purpose we fix any $\upkappa\in (\frac{sp}{p-1}, 1)$ and
$m\geq 3$ large such that 
$$\upkappa(p-1) > p-1-sp \quad \text{and}\quad (p-1)-sp<\frac{m}{m-1}\upkappa.$$
We repeat the scheme described in section~\ref{S-genstr} with $\varrho_2=\frac{3}{2}$
and $\varrho_1=1$. As before, we suppose
$\sup_{B_2\times B_2}\Phi>0$. 
We set the notation
\begin{gather*}
\delta_0=\eta_0=\varepsilon_1 (\log^2(|\abar|))^{-1}, \quad
\uptheta = \frac{n+1}{2}+p-sp, \quad
\tilde\rho=\frac{\frac{n+1}{2}+q-tq}{q-tq}, 
\\
\rho=\frac{\frac{n+1}{2}+p-sp}{p-sp},
\quad
\cone=\{z\in B_{\delta_0|\abar|} \; :\; |\langle \abar, z\rangle| \geq (1-\eta_0)|\abar||z|\},
\\
\cD_1=B_\delta\cap \cone^c\quad \cD_2= B_{\frac{1}{8}}\setminus(\cD_1\cup\cone).
\end{gather*}
Using Lemma~\ref{L3.1}(ii), Lemma~\ref{L3.2}(ii) and Lemma~\ref{L3.5}, and by choosing $\varepsilon_1$ small enough independent of $L$ we get,
 $$I_1 +I_2 + J_1 \geq \frac{C}{2} L^{p-1} |\abar|^{p-1-sp} (\log^2(|\abar|))^{-\uptheta},$$
 for all $L\geq L_0$, for some constant $L_0$.  From our calculations in Step 3, it can be easily seen that for some small $\eta>0$ the following
 hold
 \begin{gather*}
 \theta:=\frac{(p-1)-sp+\eta}{2(p-1)-sp}\in (0, 1), \quad \frac{m-1}{m}\upkappa(p-1)+ \theta ((p-1)-sp)> (p-1)-sp + \eta
 \\
 \upkappa(p-1) -\theta sp>  p-1-sp + \eta.
 \end{gather*}
 From Lemma~\ref{L4.1}, we then get  $I_3\geq -C |\abar|^{p-1-sp +\eta}$. Also, if $q\leq 2$, from Lemma~\ref{L4.2} we have
 $J_2\geq -C |\abar|^{p-1-sp +\eta}$. Suppose that $q>2$. Again, we can repeat the calculations in step 3 to find
 $\eta_1>0$ so that $J_2\geq -C |\abar|^{p-1-sp +\eta_1}$. The term $J_3$ can be estimated as follows
\begin{align*}
J_3 &\geq - C|\abar|^\alpha \int_{\tilde\delta}^{\frac{1}{8}} r^{\upkappa(q-1)-1-tq} dr
\\
&\geq - C|\abar|^\alpha \int_{\tilde\delta}^{\tilde\varrho} r^{\upkappa(p-1)-1-sp} dr
\geq -C |\abar|^{\alpha}.
\end{align*}
Now using Lemma~\ref{L3.4} together with the above estimates in \eqref{E2.7} leads to
$$ \frac{C}{2} L^{p-1} |\abar|^{p-1-sp} (\log^2(|\abar|))^{-\uptheta}\leq C_1 \max\{|\abar|^{\alpha\wedge\beta}, |\abar|^{\upkappa(p-1)},
|\abar|^{(p-1-sp)+\eta\wedge\eta_1}\}$$
for all $L\geq L_0$. Since $p-1-sp<\min\{\alpha\wedge\beta, \upkappa(p-1)\}$ and $|\abar|\to 0$ as $L\to\infty$, the above inequality can not hold
for large enough $L$, leading to a contradiction. Hence, arguing as before, we must have $u\in C^{0, 1}(\bar{B}_1)$.
This completes the proof.
\end{proof}

Now we prove Theorems~\ref{Tmain-1} and ~\ref{T1.2}

\begin{proof}[Proof of Theorem~\ref{Tmain-1}]
The proof follows from Theorems~\ref{T3.7} and ~\ref{T4.3}.
\end{proof}

\begin{proof}[Proof of Theorem~\ref{T1.2}]
If $sp>n$, using fractional Sobolev embedding \cite{DNPV} we know that $u\in C(B_2)$, whereas for $sp\leq n$, we apply \cite[Theorem~2.4]{Coz17}
to conclude that $u\in C(B_2)$. From the equivalence of weak and viscosity solution \cite{KKL19} (see also \cite[Proposition~1.5]{BT25}) it 
then follows that $u$ is a viscosity solution to $(-\Delta_p)^s u=f$ in $B_2$. Now the result follows from Theorem~\ref{Tmain-1}. Also, for the last part, we observe that when $sp>p-2$ we also have
$sp+1>p-1$.
\end{proof}

\bigskip

\subsection*{Acknowledgement}
This research of Anup Biswas was supported in part by a SwarnaJayanti
fellowship SB/SJF/2020-21/03. 

\end{document}